\documentclass{amsart}
\usepackage{amsmath}
\usepackage{amsthm}
\usepackage{amsbsy}
\usepackage{amsopn}
\usepackage{amsmath,amscd}
\usepackage{amssymb}
\usepackage{amsfonts}
\usepackage{multirow}
\usepackage[official ]{eurosym}

\newcommand{\ki}[2]{{\color{green!80!black} #1
}{\color{blue!90!black}#2 
}}

\newcommand{\ve}{\varepsilon}

\newcommand{\cI}{\mathcal{I}}

\newcommand{\bC}{\mathbb{C}}

\newcommand{\bR}{\mathbb{R}}

\newcommand{\bZ}{\mathbb{Z}}

\newcommand{\wh}{\widehat}
\newcommand{\R}{\mathbb R}
\newcommand{\eps}{\epsilon}
\newcommand{\p}{\partial}
\newcommand{\ol}{\overline}

\newcommand{\wt}{\widetilde}
\newcommand{\sign}{\mathrm{sign}}

\usepackage{graphicx}

  \newcommand{\mat}[1]{\ensuremath{
\left[\begin{matrix}#1
\end{matrix}\right]
}}

\usepackage[dvipsnames,svgnames,x11names,hyperref]{xcolor}
\usepackage{url,graphicx,verbatim,amssymb,enumerate,stmaryrd}
\usepackage[pagebackref,colorlinks,citecolor=Bittersweet,linkcolor=Bittersweet,urlcolor=Bittersweet,filecolor=Bittersweet]{hyperref}
\usepackage{tikz} 
\usetikzlibrary{arrows,decorations.pathmorphing,backgrounds,fit,positioning,shapes.symbols,chains,calc}
\usepackage[all]{xy}
\xyoption{matrix}
\xyoption{arrow}
\usepackage[hmargin=3cm,vmargin=3cm]{geometry}
\usepackage{setspace,kantlipsum}
\tikzset{help lines/.style={step=#1cm,very thin, color=gray},
help lines/.default=.5} % draws a grid spaced #1 cm

\usepackage{booktabs}

\newtheorem{theorem}{Theorem}[section]
\newtheorem{lemma}[theorem]{Lemma}
\newtheorem{proposition}[theorem]{Proposition}
\newtheorem{corollary}[theorem]{Corollary}
\newtheorem{claim}[theorem]{Claim}
\theoremstyle{definition}
\newtheorem{definition}[theorem]{Definition}
\newtheorem{inductive step}[theorem]{Inductive step}

\theoremstyle{remark}

\newtheorem{example}[theorem]{Example}
\newtheorem{inductive lemma}[theorem]{Inductive Lemma}
\newtheorem{remark}[theorem]{Remark}

\title{Geomorphology of Lagrangian ridges}

\author{Daniel  \'Alvarez-Gavela}
\address{Department of Mathematics \\ Massachusetts Institute of Technology \\ Cambridge, MA, 02139}
\email{dgavela@mit.edu}
\thanks{DA was partially supported by NSF grant DMS-1638352 and the Simons Foundation}

\author{Yakov Eliashberg }
\address{Department of Mathematics\\Stanford University \\ Stanford, CA 94305}
\email{eliash@stanford.edu}
\thanks{YE was partially supported by NSF grant DMS-1807270. }

\author{David Nadler}
\address{Department of Mathematics\\University of California, Berkeley\\Berkeley, CA  94720-3840}
\email{nadler@math.berkeley.edu}
\thanks{DN was partially supported by NSF grant DMS-1802373.}

\begin{document}
\maketitle
\begin{abstract}
We prove an ``h-principle without pre-conditions" for the elimination of tangencies of a Lagrangian submanifold with respect to a Lagrangian distribution. The main result states that such  tangencies can always be completely removed at the cost of allowing the  Lagrangian to develop certain non-smooth points, called  {\em Lagrangian ridges}, modeled on the corner $\{p=|q|\} \subset \bR^2$ together with its products and stabilizations. This result plays an essential role in the arborealization program.

 \end{abstract}
 
 \onehalfspacing
 \tableofcontents
 
 \section{Introduction}
 
  \subsection{Overview}

Let $L$ be a smooth compact Lagrangian submanifold of a symplectic manifold $(M, \omega)$.   The object of interest in this article consists of the tangencies of $L$ with respect to a field of Lagrangian planes $\gamma \subset TM$. When $\gamma$ is tangent to the fibres of a Lagrangian fibration $M \to B$ these tangencies are the same as singular points of the smooth map $L \to B$. If $\omega=d\lambda$ and $L$ is exact, then these tangencies are also the same as the singular points of the Legendrian front $\wh L \to B \times \bR$, where $\wh L$ is the Legendrian lift of $L$ in the contactization $M \times \bR$. The image $\Sigma \subset B \times \bR$ of the singular locus is known as the caustic in the literature \cite{A90}. 

In this article we present a method which allows for the complete elimination of tangencies of $L$ with respect to $\gamma$ via a geometric deformation of $L$. The precise statement is given in our main result Theorem \ref{theorem: main theorem}, which does not require any hypothesis on $\gamma$ and hence may be thought of as an  ``h-principle without pre-conditions''. 

%Although there holds an h-principle \cite{AG18a}, \cite{AG18b} which states that such tangencies can always be simplified to consist only of quadratic folds via a Hamiltonian isotopy of $L \subset M$ whenever there is no homotopy theoretic obstruction to doing so, in general the homotopy theoretic obstruction may be non-trivial. The purpose of the present article is to establish an ``h-principle without pre-conditions'' which allows for the complete elimination of such tangencies, with no hypotheses whatsoever, at the cost of allowing $L$ to develop certain non-smooth points. In other words, istead of deforming $L$ by a smooth Hamiltonian isotopy we deform $L$ by a piecewise-smooth Hamiltonian isotopy. Moreover, the failure of smoothness is restricted to a countable list of local models, finite in each dimension.

Our viewpoint is local on $L$. As the deformation will always be done in a neighborhood of the given Lagrangian $L$ we can assume that the symplectic manifold $M$ is the cotangent bundle $T^*L$ endowed with the standard symplectic form $\omega=d(pdq)$. All considered Lagrangians $Y \subset T^*L$ will be exact, i.e. $pdq|_Y=dh$, and hence could be lifted  to  Legendrian submanifolds
$ \Lambda =\{(x,h(x)),\;x\in Y\}\subset T^*L\times\R=J^1L$, where 
$J^1L$ is endowed  with the standard contact structure $\xi=\{dz-pdq=0\}$. 

% A tightly related problem concerns singularities of tangency  a Legendrian submanifold $\Lambda$ of a contact manifold $(N,\xi)$ with respect to a Legendrian distribution $\delta\subset\xi$. When $\delta$ is tangent to fibers of Legendrian fibration $\pi(N)\to B$, the  image $\pi(\Lambda)\subset B$ is called the {\em front of $\Lambda$ and the image $\pi(\Sigma)\subset B$  of the tangency locus  $\Sigma\subset \Lambda$ is called the {\em caustic}. 

%Slightly abusing the terminology we use the term ``caustic" for  the singular locus itself, even in the case of a  non-integrable $\gamma$,

Even if $\gamma$ is integrable, i.e. tangent to a Lagrangian foliation, it is well-known that $C^\infty$-generic Lagrangian tangency singularities  are in general non-classifiable, see \cite{AGV85}.  However, if  certain homotopical conditions given in terms of the homotopy class of the Lagrangian plane field $\gamma|_L$ are met, then by work of the first author \cite{AG18b} the tangencies can be reduced to the simplest ones of the so-called fold type  via a $C^0$-small Hamiltonian isotopy, see Section \ref{sec:h-principle} below. In the presence of homotopical obstructions the higher Lagrangian tangency singularities  cannot be removed by means of a Hamiltonian isotopy, so any attempt at removing them must allow for deformations of the Lagrangian submanifold more dramatic than a Hamiltonian isotopy.

As a first step towards removing Lagrangian tangencies, note that one can trade Lagrangian fold tangencies with respect to $\gamma$ for corner singularities of the Lagrangian itself, which are transverse to $\gamma$ in the sense that every Lagrangian tangent plane is transverse to $\gamma$. Namely, consider the 1-dimensional Lagrangian submanifold $\{q=p^2\} \subset \R^2$, which has a fold tangency with respect to the vertical distribution $\gamma=\{dq=0\}$. We can replace this smooth Lagrangian submanifold with the piecewise smooth Lagrangian submanifold $\{q=|p|\}$, which we call the order 1 ridge, and which is transverse to $\gamma$. See Figure \ref{foldridge} for an illustration, where it is also shown how one may interpolate between the two models while preserving exactness.

 It follows from this discussion that if the homotopical conditions for removing higher tangencies are satisfied by $\gamma|_L$, then we may completely eliminate the tangencies by replacing the fold tangencies produced by the h-principle \cite{AG18b} with order 1 ridges.

Our main result Theorem \ref{theorem: main theorem} shows more generally that by  creating certain standard combinatorial singularities called {\em ridges} one can always make a Lagrangian $L$ transverse to a Lagrangian distribution $\gamma$, even without any homotopical pre-conditions on $\gamma|_L$.  The deformation consists of two steps: first the necessary ridges are introduced via a local model and then the resulting  piecewise smooth Lagrangian is further deformed by a Hamiltonian isotopy.  %Higher order ridges are built out of the order 1 ridge by taking products and stabilizations. 
Note that a posteriori all the results can be reformulated back  in  the smooth category by smoothing the ridges.

\subsection{Lagrangian ridges and ridgy isotopies}

We now define ridgy Lagrangians and ridgy isotopies. In the standard symplectic $\bR^2=T^*\bR$ consider the subset $R=\{ pq=0 ; \, \, q  \geq 0, p \geq 0 \}$, which we call the model ridge of order 1, see Figure \ref{modelridge}.

\begin{definition} The model ridge of order $k$ in the standard symplectic $\bR^{2n}=T^*\bR^n$ is defined to be the product $R_{k,n} = R^k  \times \bR^{n-k} \subset (T^*\bR)^k \times T^*\bR^{n-k}$, $0 \leq k \leq n$, i.e the $(n-k)$-fold stabilization of $R^k=R \times \cdots \times R$ ($k$ times). \end{definition}

\begin{figure}[h]
\includegraphics[scale=0.6]{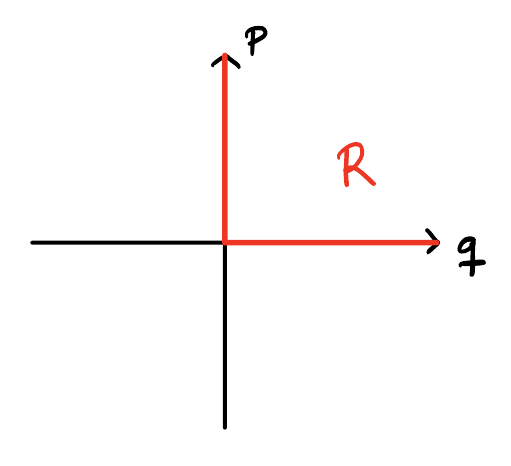}
\caption{The model ridge $R \subset T^*\bR$. Note that $R$ is the union of the half-line $\{p=0, \, q \geq 0 \}$ together with the inner conormal $\{q=0, \, p \geq 0 \}$ of its boundary point $q=0$. The model $R$ is symplectomorphic to $\{p=|q|\} \subset T^*\bR$.}
\label{modelridge}
\end{figure}

\begin{example}
The order $n$ ridge $R_{n,n} \subset T^*\bR^n$ is the union to all the inner conormals of the faces of a quadrant in $\bR^n$, hence is the union of the $2^n$ linear Lagrangians $\{ p_j=q_k =0, \, q_j, p_k \geq 0 , \, j \in I, \, k \not \in I \}$, where $I \subset \{ 1, \ldots , n \}$, see Figure \ref{conormals}.
\end{example}

\begin{figure}[h]
\includegraphics[scale=0.5]{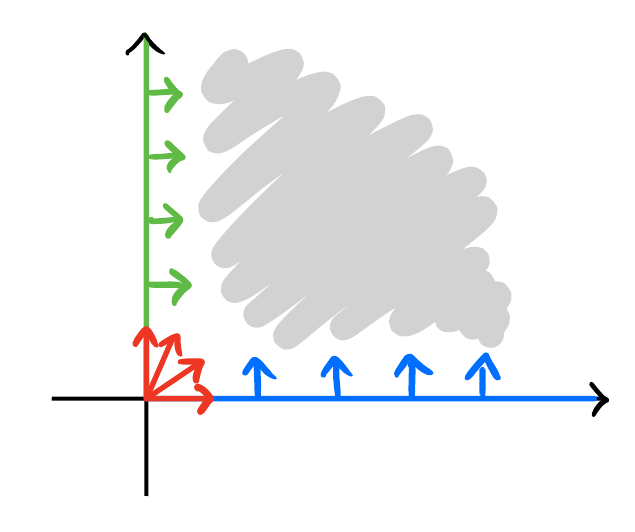}
\caption{The order 2 ridge $R_{2,2} \subset T^*\bR^2$ is the union of the inner conormals to the faces of the quadrant $\{q_1 \geq 0, \, q_2 \geq 0 \}$. In grey $\{p_1=p_2=0, \, q_1,q_2 \geq 0 \}$, in blue $\{p_1=q_2=0, \, q_1,p_2 \geq 0 \}$, in green $\{p_2=q_1=0, \, q_2,p_1 \geq 0 \}$ and in red $\{q_1=q_2=0, \, p_1,p_2 \geq 0 \}$. The model $R_{2,2}$ is symplectomorphic to $\{p=|q|\} \times \{p=|q|\}  \subset T^*\bR \times T^*\bR = T^*\bR^2$. }
\label{conormals}
\end{figure}
 
\begin{definition}
An $n$-dimensional \emph{ridgy Lagrangian} in a symplectic manifold $M$ is a closed subset $L \subset M$ which is covered by open subsets $U\subset M$ such that $(U, U \cap L)$ is symplectomorphic to some $(B, B \cap R_{k,n})$, for $B \subset \bR^{2n}$ a ball centered at the origin.
\end{definition} 

\begin{figure}[h]
\includegraphics[scale=0.62]{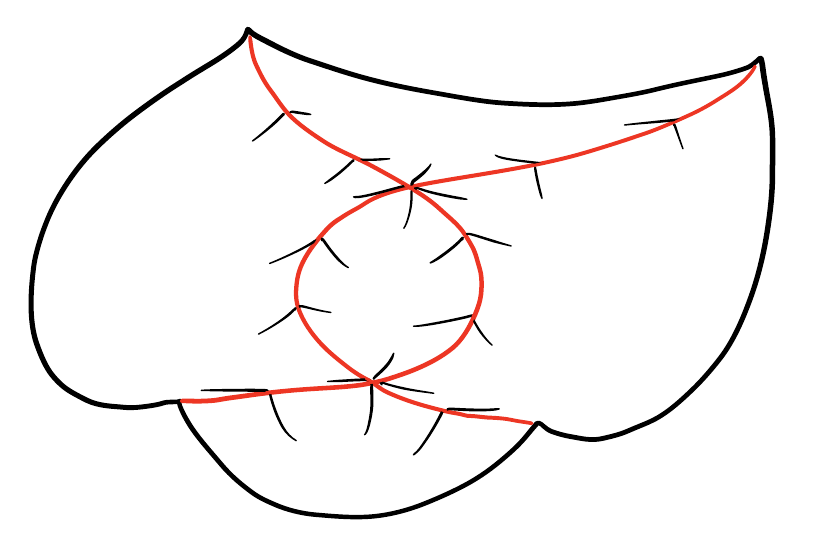}
\caption{A 2-dimensional ridgy Lagrangian has order 1 ridges along a union of immersed curves, which intersect each other (or themselves) at a discrete set of points where we have order 2 ridges.}
\label{ridgy}
\end{figure}

A ridgy Lagrangian has a natural stratification $L=R_0 \supset R_{1} \supset \cdots \supset R_n$, where $R_k$ is the locus of ridges of order $\geq k$, see Figure \ref{ridgy}. Note that the stratum  $R_k \setminus R_{k+1}$ is a smooth (open) isotropic submanifold of dimension $n-k$. 
%In a ridgy isotopy the ridge locus is allowed to bifurcate. For our purposes it will suffice to consider the following notion.

\begin{definition} Let $L$ be a smooth Lagrangian submanifold in a symplectic manifold $M$.
\begin{enumerate}
\item
Let $N_1 , \ldots , N_m \subset L$  be co-oriented separating   hypersurfaces   defined by equations $\phi_j=0$ for some $C^\infty$-functions $\phi_j:L\to\R$ without critical points on $N_j$.  We assume that the $N_j$ are co-oriented by the outward transversals to the domains $\{\phi_j\leq 0\}$. We assume that the $N_j$ are mutually transverse, i.e. each $N_j$ is transverse to all possible intersections of the other $N_i$. Denote $\phi_j^+=\max(\phi_j,0)$ and choose a cut-off function $\theta_j$ which is equal to 1 on $N_j$ and to $0$ outside a neighborhood of $N_j$.  Define a function $\Phi:L \to \bR$ (which is $C^1$ and piecewise $C^\infty$) by the formula
 $$\Phi:=\sum\limits_{j=1}^m\theta_j\left(\phi_j^+\right)^2.$$ An {\em earthquake isotopy} with faults $N_j$  is defined as a family of Lagrangians $L_t$ given by the homotopy of generating functions $t\Phi$, i.e. $L_t=\{p=td\Phi\}$, $t\geq 0$, see Figure \ref{earthquake}.
\item A {\em ridgy isotopy} is   an earthquake isotopy followed by an ambient Hamiltonian isotopy.
 \end{enumerate}
 \end{definition}
Of course, the earthquake isotopy can be realized by an ambient Hamiltonian isotopy beginning from any $t>0$, i.e. for all $\ve > 0$ there exists a Hamiltonian isotopy $\varphi_t$ such that $L_{t +\ve} = \varphi_t(L_\ve)$, $t \geq 0$. 

%The obstruction only depends on the homotopy class of $\gamma$ as a field of Lagrangian planes along $L$. 
%Moreover, even if this homotopical obstruction vanishes and the simplification of caustics is possible, the complete elimination of caustics is rarely possible. 

\begin{figure}[h]
\includegraphics[scale=0.635]{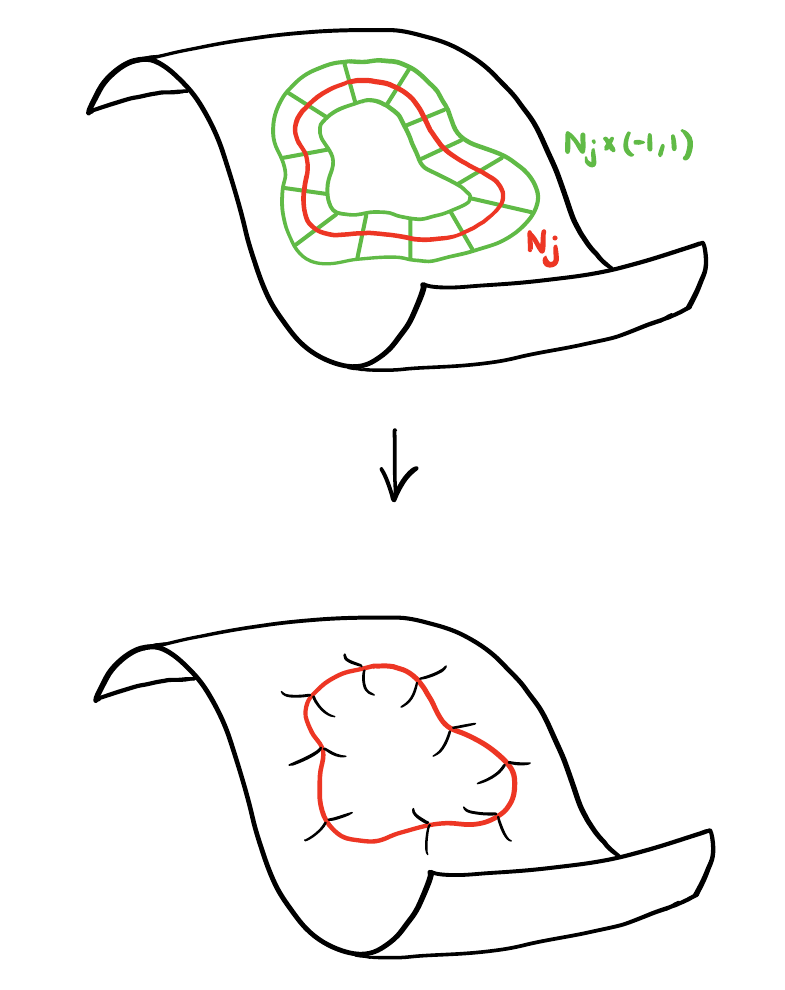}
\caption{An earthquake isotopy. Note that in general the hypersurfaces $N_j$ may intersect each other. A ridgy isotopy further deforms the result of an earthquake isotopy by a Hamiltonian isotopy.}
\label{earthquake}
\end{figure}

% relax the notion of Hamiltonian isotopy so that $L$ is allowed to develop a certain class of Lagrangian singularities called ridges, then the obstruction disappears and in fact it is always possible to completely eliminate the caustics of $L$ with respect to any Lagrangian field $\gamma$.

\subsection{Main results}

We can now state our main result, which we call the Ridgification Theorem. Recall that $L$ is a smooth, compact Lagrangian submanifold of a symplectic manifold $M$.
 
 \begin{theorem}\label{theorem: main theorem}
For any Lagrangian distribution $\gamma$ there exists a ridgy isotopy $L_t$ of $L$ such that $L_1 \pitchfork \gamma$.
 \end{theorem}
 
 To be clear: the condition $L_1 \pitchfork \gamma$ means that for any $x \in L_1$ and for any Lagrangian plane $P \subset T_xM$ tangent to $L_1$ we have $P \pitchfork \gamma_x$. If $x \in L_1$ is a ridge of order $k$, there are $2^k$ such Lagrangian planes.
 
  %In fact, this application is the main source of motivation for the results of this article, which we therefore consider as the second in a series of papers on the arborealization program. In the next installation of this series the authors will use a parametric version of the Ridgification Theorem, stated precisely in Section \ref{section: adapted version} below, to establish a uniqueness counterpart to the existence result of \cite{AGEN20b}.

\begin{remark}
Theorem \ref{theorem: main theorem} also holds in the following variants:
\begin{enumerate}
\item $C^0$-close form: we can arrange it so that the ridgy isotopy $L_t$ is $C^0$-small. This means that given a fixed but arbitrary Riemannian metric on $M$, for any $\ve>0$ we can demand that $\text{dist}(x,f_t(x) )< \ve$ for $f_t:L \to L_t$ the parametrization of the ridgy isotopy $L_t$ which is graphical during the earthquake isotopy and then is given by the ambient Hamiltonian isotopy. In particular $L_t$ stays within a Weinstein neighborhood of $L$ in $M$.
 
\item  Relative form:  if $L \pitchfork 
\gamma$ on $Op(A)$ for $A \subset L$ a closed subset then we can demand that $L_t=L$ on $Op(A)$. Here and below we use Gromov's notation $Op(A)$ for an arbitrarily small but non-specified open neighborhood of $A$. 
\end{enumerate}
\end{remark}

We will  also prove in Section \ref{section: adapted version} an  adapted version of Theorem \ref{theorem: main theorem} relative to a collar structure in the case where $L$ has boundary and corners, assuming that $\gamma$ is itself adapted to that structure. This adapted version is essential for our applications. Indeed, together with the Stability Theorem for arboreal singularities proved by the authors in \cite{AGEN20a}, the collared version of Theorem \ref{theorem: main theorem} is one of the essential ingredients in our paper  \cite{AGEN20b} on the arborealization program~\cite{N15, N17, St18}, as well as in the forthcoming work \cite{AGEN21}.

\begin{remark}
By definition, the local geometry of a ridgy Lagrangian $L$ is given by the linear models $R_{k, n}$. The space of linear Lagrangian fields $\gamma$ transverse to $R_{k, n}$ has interesting moduli without evident canonical representatives. 
\end{remark}

%In fact, the ridgy isotopy $L_t$ produced by Theorem \ref{theorem: main theorem} can be described more explicitly. It consists of the concatenation of two deformations (a) and (b).

  %of ridges is given by a piecewise $C^2$-function defined in a neighborhood of $\bigcup_j N_j$. Explicitly, if $u \in (-1,1)$ is a tubular neighborhood coordinate for one of the $N_j$ and $\alpha: N_j \to \bR$ is a nonvanishing function, consider the piecewise $C^2$-function 
%\[ f(x,u) = \text{sign}(u) \alpha(x) u^2, \qquad (x,u) \in N_j \times (-1,1) \]
%multiplied by a cutoff function so that $f=0$ for $u$ near $\pm 1$. Then the sum of all such functions for all $N_j$ generates a graphical ridgy Lagrangian in $T^*L$, which we can view as a ridgy Lagrangian in $M$ by the Weinstein neighborhood theorem. Multiplying the generating function by a parameter $\ve \in [0,1]$ gives a deformation of $L$ in which the ridges appear instantaneously from $\ve=0$ to $\ve > 0$.

%Since each $N_j$ is separating, instead of introducing ridges instantaneously along the $N_j$ we can decompose the deformation (a) into a piecewise smooth exact regular homotopy of ridgy Lagrangians in which the ridge locus experiences finitely many Morse bifurcations at isolated times. In this case the ridgy isotopy has the property that the $L_t$ are the reductions of an $(n+1)$-dimensional ridgy Lagrangian $\widehat L \subset T^*(L \times [0,1])$ over the slices $L \times t$.

 \subsection{h-principle for removing higher Lagrangian tangency singularities}\label{sec:h-principle}
 
 The problem of simplifying the tangency locus of a smooth Lagrangian submanifold $L \subset M$ with respect to a Lagrangian plane field $\gamma \subset TM$ was first studied by Entov \cite{En97}, who used the method of surgery of singularities to establish an h-principle for the class of $\Sigma^2$-nonsingular plane fields, i.e. those $\gamma$ for which $\dim(TL \cap \gamma) <2$. In \cite{AG18a} and \cite{AG18b} the methods of holonomic approximation and wrinkling were used by the first author to extend this h-principle to arbitrary Lagrangian plane fields. The simplest version of the h-principle can be formulated as follows.
 
 \begin{theorem}\label{theorem: caustics}
 Suppose that $\gamma$ is homotopic through Lagrangian plane fields to a Lagrangian plane field $\wh \gamma$ which is transverse to $L$. Then $L$ is Hamiltonian isotopic to a smooth Lagrangian submanifold $\wh L$ whose tangency singularities with respect to $\gamma$ consist only of folds.
 \end{theorem}
 
 \begin{remark}
 More generally, it is enough to assume that $\gamma$ is homotopic to a $\wh \gamma$ with respect to which $L$ only has fold tangencies.
 \end{remark}
 
 The {\em fold} is the simplest type of  singularity. In the case where $\gamma$ is integrable, the germ is given by (a stabilization of) the local model $\{ q=p^2 \} \subset T^*\bR$, where the Lagrangian field is the vertical distribution $\gamma=\{dq=0\}$. If the tangency locus of $L$ with respect to $ \gamma$ consist only of folds, then $\dim(TL \cap \gamma) \leq 1$ and the tangency locus $\Sigma = \{ \dim(TL\cap \gamma)=1\}$ is a transversely cut out smooth hypersurface in $L$. Moreover, the line field $\ell= TL \cap \gamma$ is transverse to $\Sigma$ inside $TL$. These properties completely characterize the fold (also in the non-integrable case).

 \begin{remark}
 Even in the smooth (as opposed to symplectic) category, the elimination of folds is not usually possible. Moreover, while in the smooth category the only non-trivial constraints are on the topology of the image of the fold, see \cite{G09, G10}, where the only thing which matters is that the fold locus is non-empty, see \cite{E70}, in the symplectic case there are also constraints on the topology of folds (e.g. the number of its components) in the source Lagrangian, see \cite{En98}. See also \cite{FP98, FP06} for further constraints on the caustic locus.
 \end{remark}
 
The fold is closely related to the order 1 ridge. More precisely, observe that the 1-dimensional Lagrangian model $\{q=p^2 \}$ has a fold type tangency to the vertical Lagrangian distribution $\gamma=\{dq=0\}$, while the  ridgy Lagrangian $\{ q=|p|\}$ is transverse to $\gamma$.
 \begin{comment}Let us take a cut-off function $\sigma:[0,\infty)\to[0,1]$ which is equal to $1$ on $[0,\frac12]$, equal to $0$ outside $[0,1]$ and has non-positive derivative. Define  the function $$\phi_\varepsilon(p)= \frac13\left(1-\sigma\left(\frac{|p|}{\varepsilon}\right)\right)p^3+\frac\varepsilon 2\sigma\left(\frac{|p|}{\varepsilon}\right)\sign(p)p^2$$  and set
 $$L_\varepsilon:=\left\{q =\frac{\p \phi_\varepsilon(p)}{\p p}\right\}.$$
 
 Then $L_0=  \left\{q=p^2 \right\}$ is a smooth Lagrangian with a fold tangency singularity to $\gamma$ at the origin, while $L_\varepsilon$ for any $\varepsilon>0$ is a ridgy Lagrangian transverse to 
 $\gamma$.
  Since we define the deformation at the level of generating functions, exactness is automatic, see Figure \ref{foldridge}. \end{comment}
 Let us take a cut-off function $\sigma:[0,\infty)\to[0,1]$ which is equal to $1$ on $[0,\frac12]$, equal to $0$ outside $[0,1]$. Define for $q \geq 0$  the generating function 
 $$
 z =\pm  \left( \sigma\left( \frac{q}{\eps} \right) \frac{q^2}{2 \eps} + \big( 1 - \sigma\left( \frac{q}{\eps}\right) \big) \frac{2q^{3/2}}{3} \right)
 $$
 
 which generates a Lagrangian $L_\ve \subset T^* \R$. Note that for any $\ve>0$, $L_\ve$ is a ridgy Lagrangian transverse to $\gamma$, while 
 \begin{comment}
 
 $$\phi_\varepsilon(p)= \frac13\left(1-\sigma\left(\frac{|p|}{\varepsilon}\right)\right)p^3+\frac\varepsilon 2\sigma\left(\frac{|p|}{\varepsilon}\right)\sign(p)p^2$$  and set
 $$L_\varepsilon:=\left\{q =\frac{\p \phi_\varepsilon(p)}{\p p}\right\}.$$

 \end{comment}
  $L_0=  \left\{q=p^2 \right\}$ is a smooth Lagrangian with a fold tangency singularity to $\gamma$ at the origin.% while $L_\varepsilon$ for any $\varepsilon>0$ is a ridgy Lagrangian transverse to 

Note also that because we define the deformation at the level of generating functions, exactness is automatic, see Figure \ref{foldridge}.

  This deformation can be achieved by a ridgy isotopy: it is essentially an earthquake isotopy along the fold locus (strictly speaking, it needs to be corrected by a subsequent Hamiltonian isotopy to get the symmetry about the fold locus, but this is not important).

\begin{figure}[h]
\includegraphics[scale=0.5]{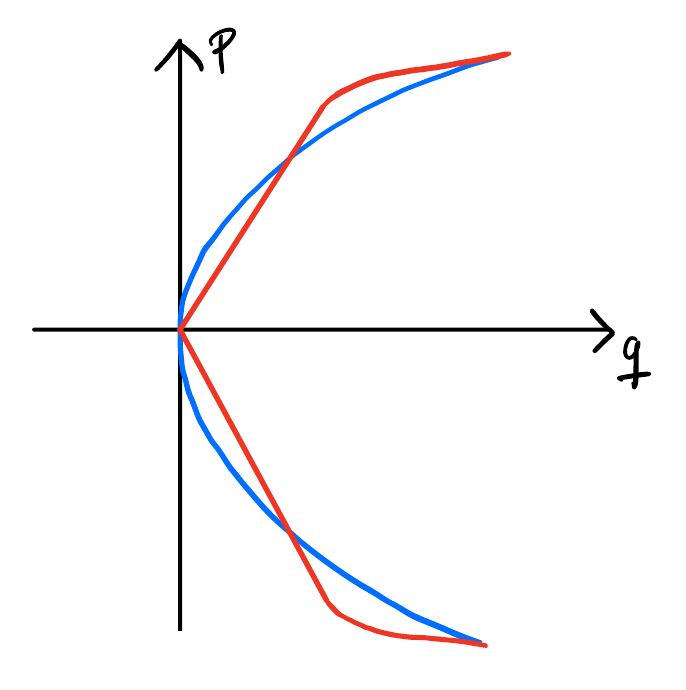}
\caption{A fold (blue) becomes a ridge (red). Exactness means that the area of the region bounded by the blue and the red curves is zero when counted with sign.}
\label{foldridge}
\end{figure}

If the deformation is performed close enough to the fold point (so that $\gamma$ is almost constant as in our local model), then the resulting ridgy Lagrangian is transverse to the Lagrangian plane field $\gamma$ with respect to which the smooth Lagrangian had a fold. Hence Theorem \ref{theorem: main theorem} is an immediate consequence of Theorem \ref{theorem: caustics} when $\gamma$ is homotopic to a Lagrangian plane field transverse to $L$.

%\begin{figure}[h]
%\includegraphics[scale=0.6]{1Dcase}
%\caption{A fold tangecy can be replaced with an order 1 transverse ridge.}
%\label{1Dcase}
%\end{figure}

Note that the above relation between folds and ridges only holds for order 1 ridges, i.e. stabilizations of the standard 1-dimensional ridge $R \subset T^*\bR$. Higher order ridges carry subtler homotopical information corresponding to the higher corank singularities $\Sigma^k$ and are necessary to overcome the homotopy theoretic obstruction to the simplification of singularities. Thus Theorem \ref{theorem: main theorem} shows the best one can do if nothing is known about the homotopy class of $\gamma$.

  \subsection{Structure of the article}
  
We begin our proof of Theorem \ref{theorem: main theorem}  by  showing  existence of  a formal solution, which is established in Section \ref{section: formal solution} by working one rank 1 form at a time. The resulting formal solution is then deformed to an integrable solution in two steps. First, in Section \ref{section: homotopically integrable solution} we align the ridge directions to the homotopy class necessary for integrability. Then in Section \ref{section: integrable solution} we integrate our formal solution and finish the proof of our main theorem. Finally, in Section \ref{section: adapted version} we explain a version adapted to a collar structure in the case where the Lagrangian has boundary and corners.
 
 \subsection{Acknowledgements}
 We are very grateful to Laura Starkston who collaborated with us on the initial stages of this project. 
% We are also grateful to John Pardon who helped us to crystalize the notion of an (n, n ? 1)-polarization and to Søren Galatius who explained to us how to compute obstructions to its existence.
The first author is grateful for the great working environment he enjoyed at the Institute for Advanced study and at Princeton University, as well as for the hospitality of CRM Montreal. 
The second author thanks RIMS Kyoto and ITS ETH Zurich for their hospitality. The third author thanks MSRI for its hospitality.
We are very grateful for the support of the American Institute of Mathematics, which hosted a workshop on the arborealization program in 2018 from which this project has greatly benefited.
We are grateful to the referee for numerous useful comments, suggestions and corrections.

\section{Formal solution}\label{section: formal solution}

 \subsection{Tectonic fields} We begin by introducing the notion of a tectonic field, which is the formal analogue of a ridgy Lagrangian. Recall that a polarization of a symplectic vector space $V$ consists of a pair of transverse linear Lagrangian subspaces $\tau, \nu \subset V$. For a fixed polarization $(\tau, \nu)$  there is a bijective correspondence between graphical linear Lagrangian subspaces of $V$ (i.e. transverse to $\nu$) and quadratic forms on $\tau$. Indeed, both can be thought of as symmetric linear maps $\tau \to \tau^*$, where by symmetric we mean equal to its own transpose under the canonical isomorphism $\tau^{**} \simeq \tau$. 
 
 We will repeatedly go back and forth between the two viewpoints. Note that given two graphical linear Lagrangian subspaces $\lambda_1, \lambda_2 \subset V$  we have  $\dim(\lambda_1 \cap \lambda_2)= \dim \ker(\lambda_1-\lambda_2)$. In particular,  $\lambda_1$ and  $\lambda_2$ are transverse if and only if $\lambda_1-\lambda_2$ is a nonsingular quadratic form on $\tau$.  
 Given a smooth manifold $L$, for any $x \in L$ there is a canonical polarization of $T_x(T^*L)$ given by $\tau=T_xL$ and $\nu=T^*_xL$. Hence we can identify graphical linear Lagrangian subspaces of $T_x(T^*L)$ with quadratic forms on $T_xL$. Via this identification, graphical Lagrangian plane fields on $T^*L$ defined along the zero section $L$ form a module over $C^\infty(L)$. %Note that two graphical linear Lagrangian subspaces $\lambda_1, \lambda_2 \subset T_q(T^*L)$ are transverse if and only if $\lambda_1-\lambda_2$ is a nonsingular quadratic form on $T_xL$. In fact, we have $\dim(\lambda_1 \cap \lambda_2)= \dim \ker(\lambda_1-\lambda_2)$.
 
\begin{remark}
By a Lagrangian plane field on $L$ we mean a field of Lagrangian planes in $T^*L$ defined along the zero section.  Similarly, by a field of quadratic forms on $L$ we will always mean a smooth family $\lambda_x$ of quadratic forms on $T_xL$, $x \in L$. This is the same as a graphical Lagrangian plane field.
\end{remark}
 
 \begin{definition} Suppose we are given  dividing, co-oriented embedded hypersurfaces $N_1, \ldots , N_k \subset L$. We assume that the $N_j$ are mutually transverse, i.e. each  $N_j$ is transverse to all possible intersections of the other $N_i$, $i \neq j$. A {\em tectonic field} $\lambda$ over $L$ with faults along $N_j$ is a   collection  of fields of quadratic forms $\lambda_Q$ over the closures $\overline{Q}$ of the components $Q\subset L\setminus\bigcup_j N_j$ such that  there exist non-vanishing 1-forms $\ell_j$ on $TL|_{N_j}$, $j=1,\dots, k,$ with the following property:
\begin{itemize}
\item 
for  any point point $x\in N_j\setminus\bigcup_j N_i$ we have
 $$\lambda_{Q_+} - \lambda_{Q_-}=  \ell_j^2,$$  where we denote by $C_\pm$ the components of $L\setminus\bigcup_j N_j$ adjacent to $x$ and where the co-orientation of $N_j$ points into $Q_+$.
 \end{itemize}

 The hypersurfaces $N_j$ are called \emph{faults}, the connected components $Q$ of $L \setminus \bigcup_j N_j$ are called \emph{plates} and the hyperplane fields $\tau_j$ are called \emph{ridge directions}. We will moreover demand that the following transversality condition is satisfied:
 \begin{itemize}
 \item[] Along each intersection $N_{j_1} \cap \cdots \cap N_{j_m}$ the ridge directions $\tau_{j_s}$, $s=1,\ldots, m$, are transverse to all possible intersections of the other ridge directions $\tau_{j_r}$, $r \neq s$.
 \end{itemize}
 
 See Figure \ref{tectonicfield}.
 
   \end{definition}
 
  \begin{figure}[h]
\includegraphics[scale=0.6]{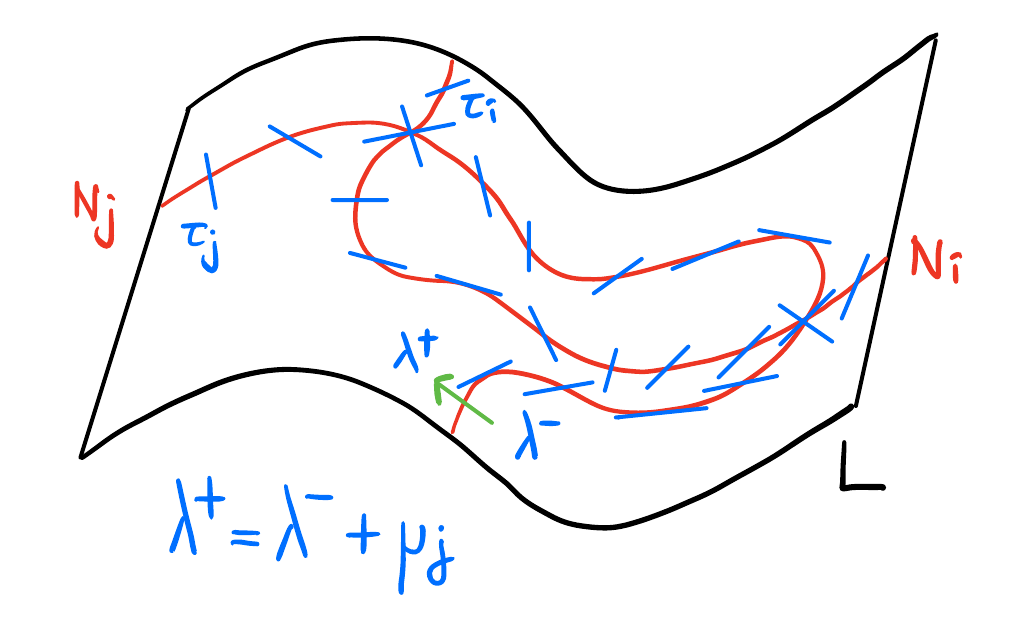}
\caption{A tectonic field. The discontinuity of $\lambda$ along the green arrow is the rank 1 form $\mu_j=\ell_j^2$ corresponding to the hypersurface $N_j$. Note that the hyperplane fields $\tau_j=\ker(\ell_j)$ need not be tangent to the $N_j$.}
\label{tectonicfield}
\end{figure}

\begin{remark}
The closure $\overline{Q} \subset L$ of each plate $Q$ of a tectonic field is a codimension zero submanifold of $L$ with boundary and corners (of any order).
\end{remark}

 \begin{remark} The hyperplane fields $\tau_j =\ker(\ell_j)$ are  co-oriented by the choice of the defining 1-forms $\ell_j$, but note that this co-orientation is not determined by the tectonic field $\lambda$ and co-orientation of $N_j$ as we could replace $\ell_j$ with $-\ell_j$.
 % In fact this condition is equivalent to (iii) up to changing the co-orientation of the $N_j$.
 \end{remark}
  
\begin{remark} Tectonic fields do not  form a module over $C^\infty(L)$, but  they can be multiplied by functions which are positive on $\bigcup_jN_j$ and can be added  when the union of the corresponding collections of faults and ridge directions satisfies the transversality conditions. For example this is vacuously satisfied when one of the tectonic fields is actually a smooth graphical Lagrangian field.
\end{remark}

%Finally, we need a notion of size for tectonic fields. The correct notion is to measure the angle they make with the horizontal distribution.

%\begin{definition} A tectonic field is $C^0$-small if the angle it makes with $TL$ is $C^0$-small.
 %\end{definition}

%\begin{remark} Note that $C^0$-smallness is the formal analogue of $C^0$-flatness for a ridgy Lagrangian.
%\end{remark}

\subsection{Formal transversalization}

The main goal of Section \ref{section: formal solution} is to prove the following transversalization result, which is the formal version of our main Theorem \ref{theorem: main theorem}. 

\begin{theorem}\label{theorem: formal transversalization}
For any Lagrangian field $\gamma$ on $L$ there exists a tectonic field $\zeta$ on $L$ such that $\zeta \pitchfork \gamma$, i.e. $\zeta|_{\ol Q} \pitchfork \gamma$ on $\ol Q$ for each plate $Q$ of $\zeta$. 
\end{theorem}

%\begin{remark}  Theorem \ref{theorem: formal transversalization} also holds in relative form: if $\gamma$ is transverse to $L$ on $Op(A)$ for $A \subset L$ a closed subset, then we can demand that $\zeta = 0 $ on $Op(A)$.
%\end{remark}

In fact we will prove the following more general extension result with $C^0$-control.

\begin{theorem}\label{theorem: formal transversalization extension} Let $\gamma$ be a Lagrangian field on $L$ and let $\zeta$ be a tectonic field on $L$. For any two disjoint closed subsets $K_1,K_2 \subset L$ there exists a tectonic field $\widehat{\zeta}$ such that the following properties hold.
\begin{itemize}
\item $\widehat{\zeta}$ is $C^0$-close to $\zeta$.
\item $\widehat{\zeta} \pitchfork \gamma$ on $Op(K_1)$.
\item $\widehat{\zeta}= \zeta$ on $Op(K_2)$.
\end{itemize}
\end{theorem}

\begin{remark} The $C^0$-closeness statement means the following. Given a fixed but arbitrary Riemannian metric on $L$, for any $\ve>0$ we can demand that the maximal angle between $\zeta$ and $\wh \zeta$ is smaller than $\ve$. Moreover, if $L$ is not compact, then the same holds for any function $\ve:L \to (0,1]$.  \end{remark}

Note that Theorem \ref{theorem: formal transversalization extension}  implies Theorem
\ref{theorem: formal transversalization} in its stronger relative form: {\em if $\zeta \pitchfork \gamma$ on $Op(A)$ for $A \subset L$ a closed subset, then we can demand that $\widehat{\zeta} = \zeta $ on $Op(A)$. } To see this take $K_2=A$ and $K_1=L \setminus Op(A)$ and use the $C^0$-closeness provided by Theorem \ref{theorem: formal transversalization extension}.

  In Section \ref{section: adapted version} we will 
prove a version of  Theorem \ref{theorem: formal transversalization extension} for the case where $L$ is a manifold with boundary and corners and $\gamma$ is adapted to the corner structure.

%\begin{definition}
%We say that a tectonic field $\zeta$ on $L \times [0,1]$ is fibered if $T[0,1] \subset \ker(\zeta)$.
%\end{definition}

%For a fibered tectonic field $\zeta$  on $L \times [0,1]$ we can think of the restriction $\zeta_t$ of $\zeta$ to $L \times \{t \}$ as a field on $L$, which generically will be tectonic for all but finitely many values of $t \in [0,1]$.

%\begin{theorem}\label{theorem: formal transversalization concordance}
%For any smooth Lagrangian distribution $\gamma$ on $L$ there exists a $C^0$-small fibered tectonic field $\zeta$ on $L \times [0,1]$ such that $\zeta_0=0$ and $\zeta_1 \pitchfork \gamma$.
%\end{theorem}

%\begin{remark} Theorem \ref{theorem: formal transversalization concordance} also holds in relative form: if $\gamma$ is transverse to $TL$ on $Op(A)$ for $A \subset L$ a closed subset, then we can demand that $\zeta_t = 0 $ on $Op(A)$ for all $t \in [0,1]$.
%\end{remark}

 \subsection{Inductive step} The key ingredient in the proof of the formal transversalization theorem is the following inductive procedure, in which we only deal with a rank 1 form at a time. 
 
 \begin{lemma}\label{inductive lemma}
 Let $\lambda, \eta$ be smooth fields of quadratic forms on $L$, with $\eta=\alpha \ell^2$ for a field of non-zero linear forms $\ell$ and a real valued function $\alpha:L \to \bR$. Let $\zeta$ be a tectonic field which is transverse to $\lambda$. Then there exists a $C^0$-small tectonic field $\zeta'$ such that $\zeta+\zeta'$ is a tectonic field transverse to $ \lambda + \eta$. If $\eta=0$ on $Op(A)$ for some closed subset $A \subset L$, we may moreover demand that $\zeta'=0$ on $Op(A)$.
 \end{lemma}

 \begin{proof} Denote by $N_1, \ldots , N_k$ the faults,  by $\tau_1, \ldots , \tau_k$ the ridge directions and by $Q_1, \ldots , Q_m$ the plates of the tectonic field $\zeta$. Note that $\eta$ has rank $1$ and hence $\lambda + \eta - \zeta$ has rank $\geq n-1$. Let $\Sigma \subset L$ denote the locus where the rank of $\lambda + \eta - \zeta$ is exactly $n-1$, i.e. $\Sigma= \{ \det(\lambda + \eta - \zeta)=0\}$, where here and below we fix an arbitrary Riemannian metric on $L$ to compute the determinant. Set $\Sigma_j = Q_j \cap \Sigma$. Our first goal is to reduce Lemma \ref{inductive lemma} to the case where the following properties hold.
 \begin{itemize}
 \item[(A)] The closure $\overline{\Sigma}_j$ of $\Sigma_j$ is a properly embedded smooth codimension $1$ submanifold with boundary and corners of $\overline{Q}_j$ such that  $\overline{\Sigma}_j$ is transverse to all intersections of the faults $N_i$, $i \neq j$ and $\overline{\Sigma}_j$ is transverse to all intersections of the other $\overline{\Sigma}_i $ along their boundary.
 \item[(B)] $\tau = \ker(\eta)$ is transverse to all possible intersections of the ridge directions $\tau_{j_1}, \ldots , \tau_{j_m}$ along the intersection of $N_{j_1} \cap \cdots \cap N_{j_m}$ with each
 %the closure of
  $\overline{\Sigma}_i$.
 \end{itemize}
 
  \begin{figure}[h]
\includegraphics[scale=0.6]{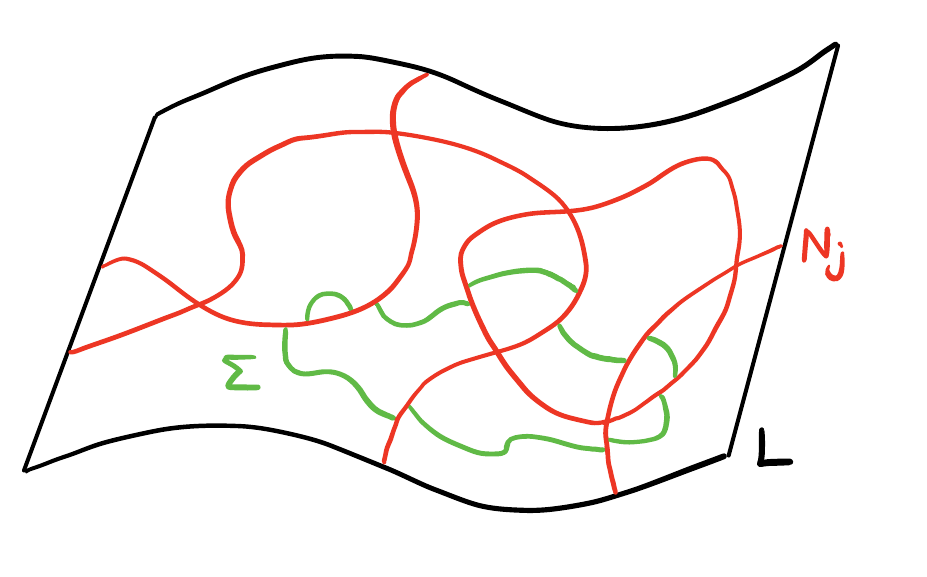}
\caption{The singular locus $\Sigma \subset L$.}
\label{singularlocus}
\end{figure}

 Suppose first that we know Lemma \ref{inductive lemma} to be true when (A) holds. Let $\lambda, \eta$ and $\zeta$ as in the statement of the lemma. By genericity of transversality we can find a $C^0$-small smooth field of quadratic forms $\varphi$ such that the hypersurfaces $\det(\lambda+\varphi+ \eta - \zeta)=0$ are transversely cut out on each plate of $\zeta$, are transverse to all intersections of the faults, and are transverse to each other along their boundaries. Then by assumption we can apply Lemma \ref{inductive lemma} with $\lambda+ \varphi$ instead of $\lambda$ (which is still transverse to $\zeta$ since $\varphi$ is $C^0$-small), obtaining a $C^0$-small tectonic field $\zeta'$ such that $\zeta+ \zeta'$ is transverse to $\lambda + \varphi + \eta$. Hence $\zeta''=\zeta'-\varphi$ is a $C^0$-small tectonic field such that $\zeta+ \zeta''$ is transverse to $\lambda + \eta$. It therefore suffices to prove Lemma \ref{inductive lemma} under the assumption that (A) holds.
 
Next, suppose that we know Lemma \ref{inductive lemma} to be true when (A) and (B) hold. Let $\lambda, \eta$ and $\zeta$ be as in the statement of the lemma and assume that (A) holds. Note that the non-transversality condition in (B) has codimension $ \geq n-m$ and the intersection has codimension $m+1$. Therefore by genericity we can find a smooth field of rank $\leq 1$ forms $\widetilde \eta = \widetilde \alpha \widetilde \ell^2$ which is $C^0$-close to $\eta$ and such that condition (B) holds if we replace $\eta$ by $\widetilde \eta$. We may assume the condition (A) still holds by openness. Then by assumption we can apply Lemma \ref{inductive lemma} with $\widetilde \eta$ instead of $\eta$. The output is a $C^0$-small tectonic field $\zeta'$ such that $\zeta+ \zeta'$ is a tectonic field transverse to $\lambda + \widetilde \eta$. Hence $\zeta''=\zeta'+ \eta - \widetilde \eta$ is a $C^0$-small tectonic field such that $\zeta+ \zeta''$ is a tectonic field transverse to $\lambda + \eta$. It therefore suffices to prove Lemma \ref{inductive lemma} under the assumption that (A) and (B) hold.

We now proceed to prove Lemma \ref{inductive lemma} under the assumption that (A) and (B) hold. %Extend each $\Sigma_j$ to a closed hypersurface $\widehat \Sigma_j \subset L$ so that $N_1 , \ldots , N_k , \widehat \Sigma_1 , \ldots , \widehat \Sigma_k$ form a transverse system of hypersurfaces, as in the definition of a tectonic field. The extension is possible because each $\Sigma_j$ is defined by an equation $\det(\lambda+ \eta - \zeta)|_{Q_j}=0$, so it suffices to extend the function $\Delta_j = \det(\lambda+ \eta - \zeta)|_{Q_j}$ to $L$. Note that $\Sigma_j$ is canonically co-oriented by the direction  in which $\Delta_j$ is increasing and we can extend this co-orientation to $\widehat \Sigma_j$ using the extension of $\Delta_j$. 
%Extend each $\Sigma_j $ to a closed hypersurface $\widehat \Sigma_j \subset L$, so that the collection $N_1 , \ldots , N_k ,\widehat \Sigma_1 ,\ldots , \widehat  \Sigma_k$ forms a transverse system of hypersurfaces. This is possible because $\Sigma_j$ is defined by the equation $\det(\lambda+ \eta - \zeta)|_{Q_j}=0$, so it suffices to extend the function $\Delta_j = \det(\lambda+ \eta - \zeta)|_{Q_j}$ to $L$. A generic extension provides the desired transversality.
%Note that $\Sigma_j$ is canonically co-oriented by the direction in which $\Delta_j$ is increasing and hence we can extend this co-orientation to $\widehat \Sigma_j$ using the extension of $\Delta_j$. 

Extend $\Sigma_j $ to a closed hypersurface $\widehat \Sigma_j \subset L$, so that the collection $N_1 , \ldots , N_k ,\widehat \Sigma_1 , \ldots , \wh \Sigma_k$ forms a transverse system of hypersurfaces, as in condition (A). This is possible because $\Sigma_j$ is defined by the equation $\det(\lambda+ \eta - \zeta)|_{Q_j}=0$, so it suffices to extend the function $\Delta_j= \det(\lambda+ \eta - \zeta)|_{\ol Q_j}$ to $L$. A generic extension provides the desired transversality.
Note that $\Sigma_j$ is canonically co-oriented by the direction in which $\Delta_j$ is increasing and hence we can extend this co-orientation to $\widehat \Sigma_j$ using the extension of $\Delta_j$. 

We will construct the tectonic field $\zeta'$ inductively, working plate by plate.  We begin with the first plate $Q_1$.

Fix a tubular neighborhood $U_1=\widehat \Sigma_1 \times (-1,1)$ of $\widehat \Sigma_1$ with coordinates $(x,u_1)$ so that $\partial_{u_1}$ agrees with the specified co-orientation of $\Sigma_1$. Write $\eta=\alpha \ell^2$ as in the statement of the lemma. Fix a cutoff function $\psi:[0,1] \to [0,1]$ such that $\psi=1$ near $0$ and $\psi=0$ near $1$. Since $\lambda - \zeta$ is nonsingular, the restriction of $\lambda + \eta - \zeta$ to $\tau=\ker(\ell)$ is nonsingular, where we recall $\eta = \alpha \ell^2$ for $\ell$ a non-vanishing 1-form. Let $\delta \in \{ \pm 1\}$ be the sign of the determinant of $(\lambda+\eta-\zeta)|_\tau$ on $Q_1$. Pick $\ve_1>0$ arbitrarily small and consider the tectonic field $\zeta_1^{\ve_1}$ given by
\[ \zeta_1^{\ve_1}=  - \delta \psi_{\ve_1}(u_1) \ell^2, \qquad \psi_{\ve_1}(u_1)= \ve_1\text{sign}(u_1)\psi(|u_1|/\ve_1). \]

\begin{remark} Note that $\zeta_1^{\ve_1}$ is a tectonic field with fault $\wh \Sigma_1$ and $C^0$-norm proportional to $\ve_1$. 
\end{remark}

\begin{claim}\label{claim1}
If $\ve_1$ is chosen small enough, then $\zeta + \zeta^{\ve_1}_1$ is transverse to $\lambda + \eta$ on $Q_1$.
\end{claim}
\begin{proof}[Proof of Claim \ref{claim1}] Fix an arbitrary point in $\Sigma_1$. Choose a local frame $\kappa_1, \ldots , \kappa_{n-1}$ of $\tau^*$ with the corresponding $n(n+1)/2$ quadratic forms $\ell^2, \kappa_j^2$, $(\kappa_j + \ell)^2$, $(\kappa_j + \kappa_i)^2$, $i,j=1, \ldots , n-1$, $i<j$. By considering the symmetric matrix which corresponds to the frame $(\kappa_1, \ldots, \kappa_{n-1}, \ell)$ of $T_x^*L$ we can compute the determinant $\det(\lambda+ \eta -\zeta)$ to be of the form $A(x,u_1)f_1(x,u_1) + B(x,u_1)$, $x \in \Sigma_1$, $u_1 \in (-1,1)$, where $A$ is a non-vanishing function, namely the complementary minor corresponding to the forms $(\kappa_j+\kappa_i)^2$.
% and $f_1$ is a function which satisfies $\partial_{u_1} f(x,0)>0$. 

That $\Sigma_1$ is cut out transversely means $\partial_{u_1} \det( \lambda + \eta - \zeta)  > 0$ at $u_1=0$, so by taking $\ve_1$ small enough we may assume that this holds for all $u_1 \in (-\ve_1,\ve_1)$ and hence $\det( \lambda + \eta - \zeta)$ is a strictly increasing function of $u_1$ in the tubular neighborhood $\ve U_1:= \wh \Sigma \times (-\ve_1,\ve_1)$ where $\zeta_1^{\ve_1}$ is supported, see Figure \ref{regulargraph}.

 \begin{figure}[h]
\includegraphics[scale=0.5]{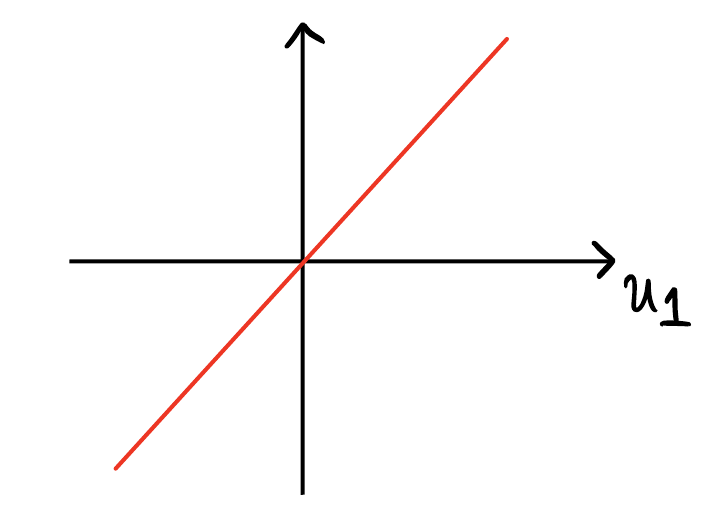}
\caption{We may assume $\partial_{u_1} \det(\lambda + \eta - \zeta)>0$ for $u_1 \in (-1,1)$.}
\label{regulargraph}
\end{figure}

Moreover, with respect to that same frame we can write $\det(\lambda+ \eta -\zeta-\zeta^{\ve_1}_1)$ in the form $A(x,u_1)(f_1(x,u_1) \pm \psi_{\ve_1}(u_1)) + B(x,u_1)$, where $\pm$ is the sign $\delta$ of $A$. Hence we have
\[ \det(\lambda + \eta - \zeta - \zeta_1^{\ve_1}) = \det(\lambda + \eta - \zeta) + |A(x,u_1)| \psi_{\ve_1}(u_1), \]
which is bounded away from zero, see Figures \ref{jump} and \ref{sum}. %Transversality elsewhere on $Q_1$ is ensured by taking $\ve_1$ sufficiently small. %However, $\zeta_1^{\ve_1}$ is not yet integrable, since its kernel $\tau$ need not be tangent to $\Sigma_1$. Our next task is to modify $\zeta^{\ve_1}_1$ to make it integrable.
\end{proof}

 \begin{figure}[h]
\includegraphics[scale=0.5]{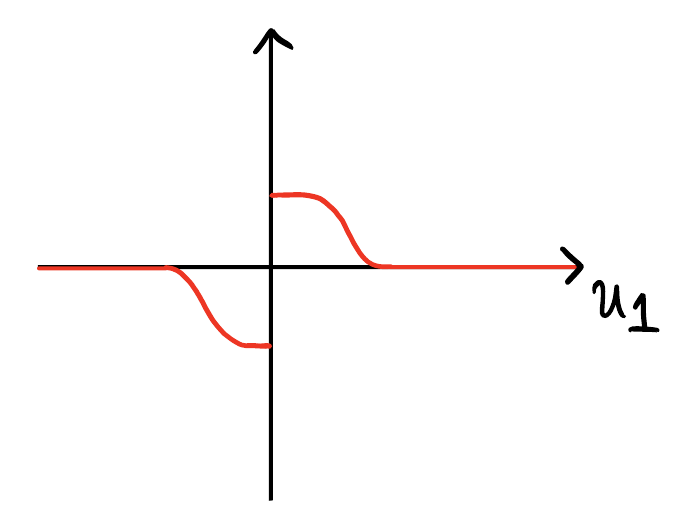}
\caption{The function $|A(x,u_1)|\psi_{\ve_1}(u_1)$, which has a discontinuity at $u_1=0$.}
\label{jump}
\end{figure}

 \begin{figure}[h]
\includegraphics[scale=0.5]{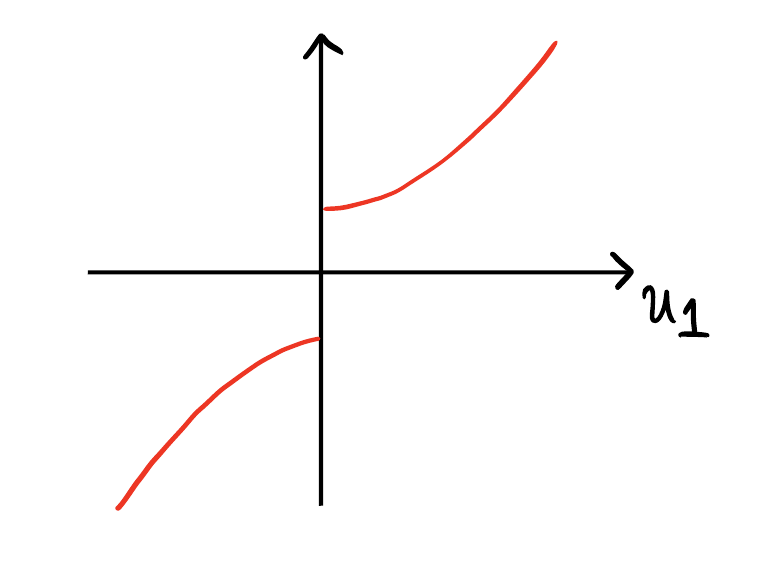}
\caption{The function $\det(\lambda + \eta - \zeta-\zeta^{\ve_1}_1)$, which is the sum of the functions $\det(\lambda + \eta - \zeta)$ and  $|A(x,u_1 )|\psi_{\ve_1}(u_1)$ illustrated in Figures \ref{regulargraph} and \ref{jump}.}
\label{sum}
\end{figure}

%We use the results and notation of Section \ref{section: cusping and capping}. Apply Theorem \ref{theorem: cuspidification} to $N=\Sigma_1$ and $H=\tau$ and then open the cusps of the resulting cuspidal hypersurface as in Lemma \ref{lemma: resolution} to obtain $ \Sigma_1' \subset L$, an immersed hypersurface with boundary. Fix an auxiliary Riemannian metric on $L$. Assume without loss of generality that $|| \ell ||=1$ with respect to this metric, where we recall $\eta = \alpha \ell^2$ for $\ell$ a field of linear forms and $\alpha$ a real valued function on $L$. Let 

Before proceeding with the inductive process on the next plate we examine the new singular locus $\Sigma_j^{\ve_1}=\det(\lambda+ \eta-\zeta-\zeta^{\ve_1}_1)=0$ on $Q_j$ for $j>1$. If $\Sigma_j \cap \widehat \Sigma_1= \varnothing$, then for $\ve_1$ small enough this singular locus is just $\Sigma_j$ and nothing changes. Suppose however that $P_j=\Sigma_j \cap \widehat \Sigma_1 \neq \varnothing$. After the addition of $\zeta^{\ve_1}_1$ to $\zeta$ the hypersurface $\widehat \Sigma_1$ becomes a fault, which causes $\Sigma_j$ to disconnect along $P_j$. The crucial observation is the following.

\begin{claim}\label{claim2} The new singular locus $\Sigma_j^{\ve_1}$ is displaced in opposite directions on each side of the fault $\widehat \Sigma_1$ and hence intersects $\widehat{\Sigma}_1$ in two disjoint parallel copies of $P_j$ in $\widehat \Sigma_1$. \end{claim}

\begin{proof}[Proof of Claim \ref{claim2}]
To verify the claim, choose a tubular neighborhood $U_j= \ol \Sigma_j \times (-1,1)$ of $ \ol \Sigma_j$ in $\ol Q_j$ with coordinates $(x,u_j)$ such that $\partial_{u_j}$ agrees with the specified co-orientation of $\Sigma_j$. We moreover assume compatibility along the boundary, i.e. that $\partial \Sigma_j \times (-1,1) \subset \partial Q_j$.

Together with the coordinate $u_1$ of the tubular neighborhood $U_1$ of $\widehat \Sigma_1$ this gives us coordinates $(x,u_1,u_j)$ of a tubular neighborhood of $P_j$ in $Q_j$. We again assume compatibility with the boundary and corner structure of $\ol Q_j$.

Near $P_j$ we can write $\det(\lambda + \eta - \zeta-\zeta^{\ve_1}_1)$ as before in the form $A(y,u_1,u_j)(f_1(y,u_1,u_j) \pm \psi_{\ve_1}(u_1)) + B(y,u_1,u_j)$, $y \in P_j$, $u_1,u_j \in (-1,1)$, where $A$ is a nonvanishing function and $\pm$ is the sign $\delta$ of $A$. In terms of our previous notation $x=(y,u_j)$. The hypersurface $\Sigma^{\ve_1}_j$ is cut out by the equation 
\[ \det(\lambda + \eta - \zeta - \zeta_1^{\ve_1}) = 0\] 
which is equivalent to
\[  \det(\lambda + \eta - \zeta)  = - |A(y,u_1,u_j)| \psi_{\ve_1}(u_1). \]
That $\Sigma_j$ is cut out transversely means that $\partial_{u_j}\det(\lambda+ \eta - \zeta)>0$ along $\Sigma_j$, so we may assume that this condition holds in the tubular neighborhood. Solving for $u_j$, the implicit function theorem implies that on each side of $\wh \Sigma_1$ the above equation cuts out a smooth hypersurface which is graphical over $\Sigma_j$. Moreover, the intersection of these hypersurfaces with $\wh \Sigma_1=\{u_1=0\}$ is given by the equations
\[ \det(\lambda + \eta - \zeta) =  |A(y,0,u_j)| \ve_1 , \qquad \det(\lambda + \eta - \zeta) = -|A(y,0,u_j)| \ve_1, \]
coming from $u_1<0$ and $u_1>0$ respectively. Since $\det(\lambda+ \eta - \zeta)$ is a strictly increasing function of $u_j$ on $U_1 \cap U_j$ which vanishes at $u_j=0$, these solutions have strictly positive and strictly negative $u_j$ coordinates respectively. Let  $u^+_j(y)>0$ and $u_j^-(y)<0$ be these coordinates, as functions of $y \in P_j$. Then $ \bigcup_y y \times 0 \times [ u^-_j(y), u^+_j(y) ]$ is a tubular neighborhood of $P_j = \bigcup_j y \times 0 \times 0$ in $\wh \Sigma_1 = \{ u_1=0 \}$ with boundary $\Sigma^{\ve_1}_j \cap \wh \Sigma_1$, which was to be proved. \end{proof}

 \begin{figure}[h]
\includegraphics[scale=0.5]{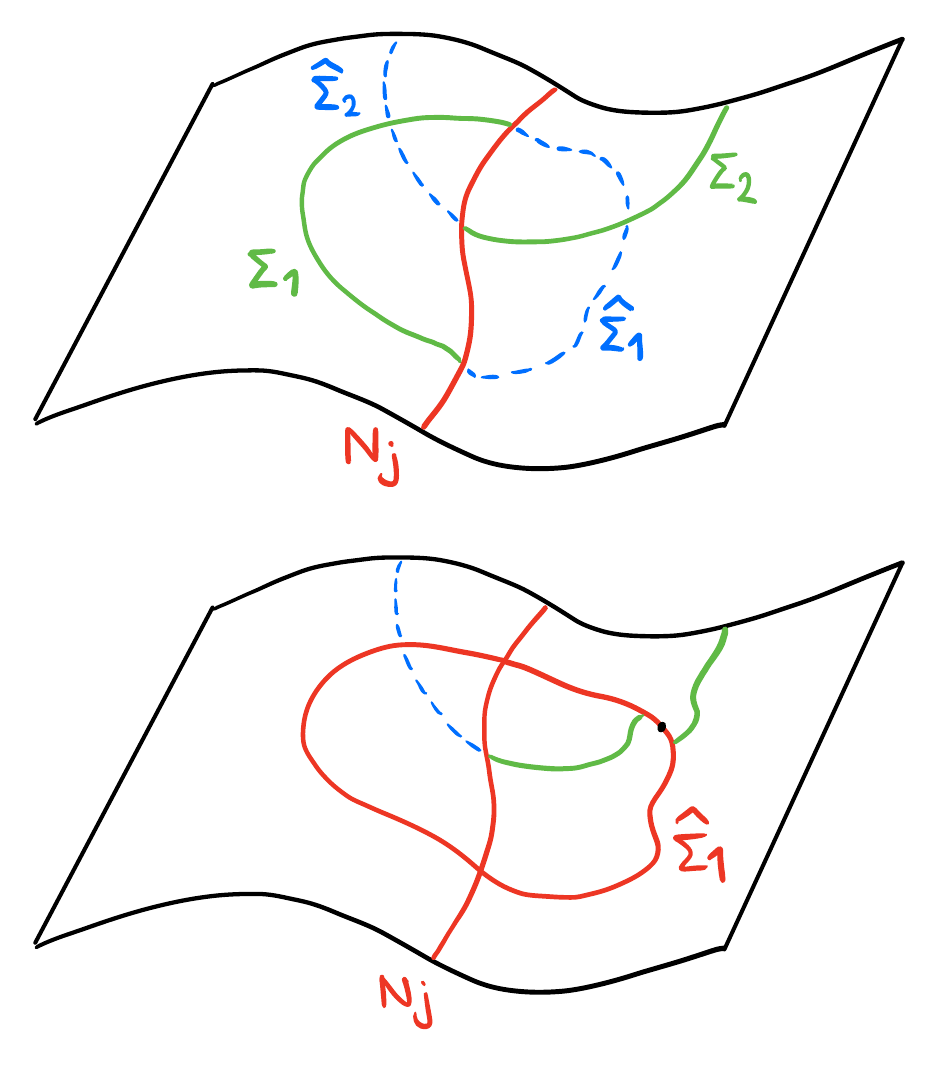}
\caption{The singular locus changes after the first step of the inductive process.}
\label{secondstep}
\end{figure}

We now reconnect $\Sigma^{\ve_1}_j$ back together along $P_j$ in the $(u_1,u_j)$ plane by parametrically closing up the family of broken curves cut out by $\det(\lambda +\eta -\zeta -\zeta_1^{\ve_1})=0$. For each fixed $y \in P_j$ we know that the interval $I_y=0 \times [u_j^{-}(y),u_j^{+}(y)]$ is disjoint from $T_y=\Sigma_j^{\ve_1} \cap (y \times (-1,1)_{u_1} \times (-1,1)_{u_j} ) \subset (-1,1)^2$ except at its endpoints $\{(0,u^{\pm}_j(y))\}=\partial T_y$. Moreover, at these boundary points $T_y$ is transverse to the vertical axis $u_1=0$, see Figure \ref{broken}. 

 \begin{figure}[h]
\includegraphics[scale=0.5]{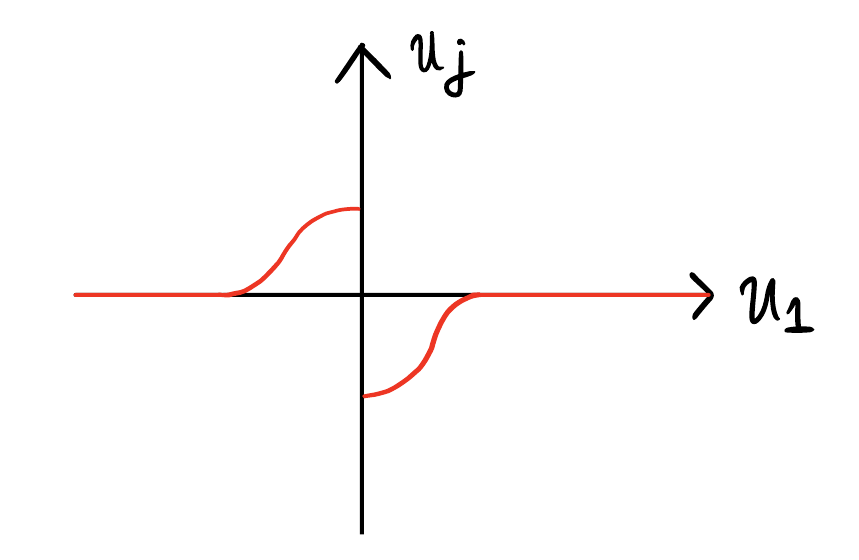}
\caption{The discontinuity of $\Sigma^{\ve_1}_j$ along $\wh \Sigma_1=\{u_1=0\}$.}
\label{broken}
\end{figure}

Consider a parametric family of smoothings $S_y$ of $T_y \cup I_y$, fixed on $T_y$. This exists because the space of smoothings of any fixed $T_y \cup I_y \subset (-1,1)^2$, fixed outside of a compact subset, is contractible. %$S_y \subset (-1,1)_{u_1} \times (-1,1)_{u_j}$ and which reconnect the point $(0,u_j^-(y))$ with the point $(0,u_j^+(y))$ in the complement of $T_y$ so that $S_y = T_y \cup S_y$ is a smooth family of boundaryless curves, see Figure \ref{Scurve}. For example one could choose
%\[ S_y(t) = \left( \sin(2 \pi t)  , (1-t)u_j^+(y)  +  tu^j_-(y)  \right) , \qquad t \in [0,1]\]
%after smoothing at the points $(0, u^j_{\pm}(y))$. 
We obtain a smooth extension of $\Sigma^{\ve_1}_j$ in $Q_j$ to a smooth hypersurface which for generic $S_y$ satisfies the required transversality conditions with respect to the faults. Since our tubular neighborhoods were chosen compatibly with the boundary structure of $\ol Q_j$ this stitching up extends to the closure, with the transversality conditions along the boundary also achieved by a generic choice of $S_y$.

Abusing notation, we denote the new, extended, hypersurface $\Sigma^{\ve_1}_j \subset Q_j$ by the same symbol. Now, $\Sigma^{\ve_1}_j$ does not agree with $\Sigma_j$ along $\partial Q_j$, hence the old extension $\wh \Sigma_j$ of $\Sigma_j$ to the rest of $L$ must be modified in order to obtain an extension $\wh \Sigma^{\ve_1}_j$ of $\Sigma^{\ve_1}_j$. To construct this modification, one applies once again the contractibility of the space of smoothings of $T_y \cup I_y$ and the genericity of the transversality condition.

%Recall that $\Sigma_j$ is given as the zero-set of a smooth function on $\ol Q_j$, namely $\det(\lambda+ \eta - \zeta)|_{\ol Q_j}=0$, and the old $\Sigma^{\ve_1}_j$ was also given as the zero set $\det(\lambda+ \eta - \zeta - \zeta^{\ve_1}_1)|_{\ol Q_j}=0$. The new $\Sigma^{\ve_1}_j$ is not immediately given as such, but it is easy to modify the function $\det(\lambda+ \eta - \zeta)|_{\ol Q_j}$ to a smooth function $\ol Q_j \to \bR$ so that the new $\Sigma^{\ve_1}_j$ is the zero set of this function. Indeed, we just need to consider the plane $(-1,1)^2$ and construct a family of functions $f_y:(-1,1)^2 \to \bR$ for which the smoothings $S_y$ are regular zero sets, and which agree with $\det(\lambda+ \eta - \zeta)|_{y \times (-1,1)^2}$ outside of a compact subset of $(-1,1)^2$. 

%To achieve this, observe that the smoothings $S_y$ are all isotopic to the zero set $\{ u_2 = 0 \} \subset (-1,1)^2$ of $\det(\lambda+ \eta - \zeta)|_{y \times (-1,1)^2}$, with the isotopy fixed outside of a compact subset of $(-1,1)^2$. Since moreover the choice of isotopy is contractible (the smoothing is unique up to contractible choice) we may construct the family $f_y$ as desired.

 \begin{figure}[h]
\includegraphics[scale=0.6]{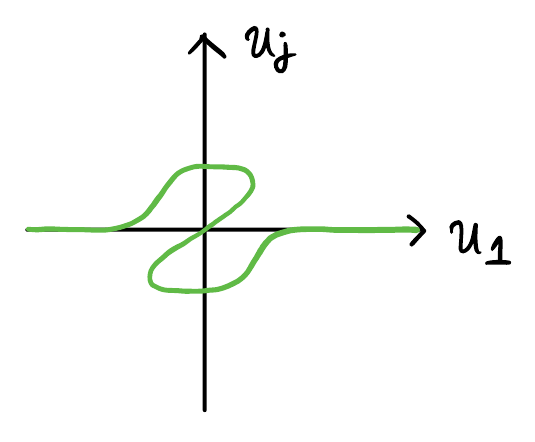}
\caption{There is a homotopically canonical way of smoothing out $T_y \cup I_y$ to $S_y$.}
\label{Scurve}
\end{figure}

%At the boundary of $Q_j$ this new hypersurface does not match up with the old extension $\widehat \Sigma_j$ of $\Sigma_j$, but there is a homotopically canonical way to let the `S' curves die out just outside of $Q_j$ so that they do match up away from $Op(Q_j)$, see Figures \ref{Sshaped1} and \ref{Sshaped2}. Hence the reconnected $\Sigma_j^{\ve_1}$ extends to a closed embedded hypersurface in $L$ which abusing notation we denote by the same symbol $\widehat \Sigma_j$, and which we may assume satisfies the required transversality conditions with respect to the faults.

% \begin{figure}[h]
%\includegraphics[scale=0.6]{Sshaped1}
%\caption{To let the `S' shaped curves die out first push the lobes of the `S' in opposite directions so that the curves become graphical over the $u_1$ axis.}
%\label{Sshaped1}
%\end{figure}

% \begin{figure}[h]
%\includegraphics[scale=0.6]{Sshaped2}
%\caption{Then one can linearly interpolate along the $u_j$ direction until the curves collapse to the $u_1$ axis.}
%\label{Sshaped2}
%\end{figure}

We now continue on to plate $Q_2$. %Extend $\Sigma^{\ve_1}_2$ to a closed hypersurface $\wh \Sigma^{\ve_1}_2 \subset L$ so that the collection $N_1, \ldots , N_2, \wh \Sigma_1, \wh \Sigma^{\ve_1}_2, \Sigma^{\ve_1}_3 , \ldots , \Sigma^{\ve_1}_k$ forms a transverse system of hypersurfaces, as in condition (A). This is possible because $\Sigma^{\ve_1}_2$ is defined by an equation $f=0$ for a smooth function $f: \ol Q_2 \to \bR$ and hence it suffices to extend $f$ to $L$. A generic extension provides the desired transversality. As before $\wh \Sigma^{\ve_1}_2$ is co-oriented.
Choose a tubular neighborhood $U_2 = \wh \Sigma^{\ve_1}_2 \times (-1,1)$. %(which by necessity will be much smaller)%. 
Again we have coordinates $(x,u_2)$. Pick $\ve_2 \ll \ve_1$ and define
\[ \zeta^{\ve_2}_2 = \pm \psi_{\ve_2}(u_2)\ell^2, \]
where the sign is determined as above. If $\ve_2$ is chosen small enough, then $\zeta+ \zeta_1^{\ve_1} + \zeta_2^{\ve_2} $ is still transverse to $\lambda + \eta$ on $Q_1$, since transversality is an open condition. Along $\Sigma^{\ve_1}_{2}$ itself we also achieve transversality by the computation carried out in the first step. Hence $\zeta+ \zeta^{\ve_1} + \zeta^{\ve_2}$ is transverse to $\lambda+ \eta$ on $Q_2$. We have thus achieved transversality on $Q_1 \cup Q_2$. 

Observe that on $Q_j$, $j>2$, the new singular locus $\Sigma_j^{\ve_1,\ve_2}=\det(\lambda+ \eta - \zeta - \zeta^{\ve_1}_1 - \zeta^{\ve_2}_2)=0$ will split along the intersection of the old singular locus $\Sigma^{\ve_1}_j=\det(\lambda+ \eta - \zeta - \zeta^{\ve_1}_1)=0$ with $\widehat{\Sigma}^{\ve_1}_2$. We proceed just like before, reconnecting and extending this new singular locus in a homotopically canonical way to a closed hypersurface $\wh \Sigma_j^{\ve_1,\ve_2}$ which contains $\Sigma^{\ve_1,\ve_2}_j$. We can then keep on going with the inductive process until we get to the last stage, which results in a $C^0$-small tectonic field $\zeta=\zeta^{\ve_1}_1+ \cdots + \zeta^{\ve_m}_m$ satisfying the required properties. This completes the proof.  \end{proof}

  \begin{remark}
 The proof of Lemma \ref{inductive lemma} automatically gives the relative form: if $\eta=0$ on $Op(A)$ for $A \subset L$ a closed subset, then we can demand that $\zeta'=0$ on $Op(A)$.
 \end{remark}
 
 \subsection{Extension step} In this section we use the inductive lemma \ref{inductive lemma} to prove the formal transversalization theorem \ref{theorem: formal transversalization extension}. The main point is that any quadratic form is a sum of rank 1 forms. First we prove a local version of the result, which we will then globalize.
 
 \begin{lemma}\label{lemma: local extension} Let $\gamma$ be a smooth field of quadratic forms on the open unit ball $B \subset \bR^n$, let $\xi$ be a tectonic field on $B$ and let $\widetilde  B \subset B$ be a smaller ball whose closure is contained in $B$. There exists a $C^0$-small tectonic field $\zeta$ such that $\xi + \zeta$ is a tectonic field transverse to $\gamma$ on the closure of $\widetilde{B}$ and such that $\zeta=0$ near $\partial B$.
 \end{lemma}
 
 \begin{proof} Fix a smooth field of quadratic forms $\sigma$ on $B$ which is transverse to $\xi$. This is always possible, for instance we can take $\sigma$ to be almost vertical. Write the difference $\gamma-\sigma$ as a sum $\alpha_1 \ell_1^2 + \cdots + \alpha_N \ell_N^2$, where the $\ell_j$ are smooth fields of linear forms. For example we can use the linear forms $X_i+X_j$, where $1 \leq i \leq j \leq n$. Then $N=n(n+1)/2$ and the identity $X_iX_j = \frac{1}{2}\big( (X_i+X_j)^2-X_i^2-X_j^2\big)$ ensures that such a decomposition exists. Let $\widetilde \alpha_j$ be a function which is equal to $\alpha_j$ on the closure of $\widetilde B$ and is equal to zero near $\partial B$.
 
 We begin by applying Lemma \ref{inductive lemma} to $\lambda = \sigma$, $\eta=\widetilde \alpha_1 \ell_1^2$ and $\zeta=\xi$. We obtain a $C^0$-small tectonic field $\zeta_1$ such that $\sigma + \widetilde \alpha_1 \ell_1^2 - \xi - \zeta_1$ is nonsingular. Next we apply Lemma \ref{inductive lemma} to $\lambda= \sigma + \widetilde \alpha_1 \ell_1^2$, $\eta=\widetilde \alpha_2 \ell_2^2$, and $\zeta=\xi + \zeta_1$. We obtain a $C^0$-small tectonic field $\zeta_2$ such that $\sigma + \widetilde \alpha_1 \ell_1^2 + \widetilde \alpha_2 \ell_2^2 - \xi - \zeta_1 - \zeta_2$ is nonsingular. We repeat this process inductively. When at the last step we apply Lemma \ref{inductive lemma}, we obtain a $C^0$-small tectonic field $\zeta=\zeta_1 + \cdots + \zeta_N$ such that $\sigma + \sum_{j=1}^N \widetilde \alpha_j \ell_j^2 - \xi -\zeta$ is nonsingular. In particular $\gamma-\xi-\zeta$ is nonsingular on $\widetilde B$. Moreover, since each time we apply Lemma \ref{inductive lemma} we have $\eta=0$ near $\partial B$, we may apply the relative version of the lemma and hence assume that $\zeta_j=0$ near $\partial B$ for each $j=1,\ldots , N$. Therefore $\zeta=0$ near $\partial B$ also. \end{proof}

 \begin{proof}[Proof of Theorem \ref{theorem: formal transversalization extension}]
For $C>0$ we set $\Omega_C=\{ x \in L: \, \, \gamma \pitchfork T_x^*L \, \, \text{and} \, \, |\det(\gamma_x)|<C\} \subset L$. Choose $C$ sufficiently large so that $\zeta \pitchfork \gamma$ outside of $\Omega_C$. Let $B_1 , \ldots , B_m$ be a cover of $\Omega_C \cap K_1$ by open balls $B_j$ such that $\overline{B}_j \subset \Omega_{2C} \setminus K_2$. In particular $\gamma$ is graphical on each $B_j$, hence can be thought of as a field of quadratic forms. Take slightly smaller balls $\widetilde B_j$ whose closure is contained in $B_j$ and such that the collection $\widetilde B_1 , \ldots , \widetilde B_m$ still covers $\Omega_C\cap K_1$. We will construct the desired $\zeta$ inductively, one $B_j$ at a time.
 
 First apply Lemma \ref{lemma: local extension} on $B_1$ to $\gamma$ and $\xi=\zeta$, producing a $C^0$-small tectonic field $\zeta_1$ such that $\zeta_1=0$ near $\partial B_1$ and such that $\zeta+\zeta_1$ is transverse to $\gamma$ on $\widetilde B_1$. Suppose that we have constructed $C^0$-small tectonic fields $\zeta_1 , \ldots , \zeta_k$ supported on $\bigcup_{j=1}^k B_j$ such that $\zeta+\sum_{j=1}^k \zeta_j$ is transverse to $\gamma$ on $\bigcup_{j=1}^k \widetilde B_j$. Apply Lemma \ref{lemma: local extension} on $B_{k+1}$ to $\gamma$ and $\xi = \zeta+ \sum_{j=1}^k \zeta_k$ to obtain a $C^0$-small tectonic field $\zeta_{k+1}$ such that $\zeta_{k+1}=0$ near $\partial B_{k+1}$ and such that $\zeta+\sum_{j=1}^{k+1} \zeta_j$ is transverse to $\gamma$ on $\widetilde B_{k+1}$. Since transversality is an open condition, by taking $\zeta_{k+1}$ to be sufficiently $C^0$-small we can ensure that $\zeta+\sum_{j=1}^{k+1}\zeta_j$ is also transverse to $\gamma$ on $\bigcup_{j=1}^k \widetilde B_j$. Hence $ \zeta+\sum_{j=1}^{k+1} \zeta_j$ is transverse to $\gamma$ on $\bigcup_{j=1}^{k+1} \widetilde B_j$ and the inductive procedure can continue. 
 
 At the last stage of the inductive procedure we obtain a tectonic field $\zeta'=\zeta+\sum_{j=1}^m \zeta_j$ which is $C^0$-close to $\zeta$, such that $\zeta' \pitchfork \gamma$ on $\Omega_C\cap K_1$ and such that $\zeta'=\zeta$ outside of $\Omega_{2C} \setminus K_2$. If $\zeta'$ is sufficiently $C^0$-close to $\zeta$ then $\zeta' \pitchfork \gamma$ also on $\Omega_{2C}\setminus \Omega_C$, because we chose $C>0$ so that $\zeta \pitchfork \gamma$ in that region. Hence $\zeta' \pitchfork \gamma$ on $K_1$ and $\zeta'=\zeta$ on $Op(K_2)$. This completes the proof.  \end{proof}

 \section{Alignment of ridges}\label{section: homotopically integrable solution}
 
\subsection{Aligned transversalization}

Let $L$ be a smooth manifold and let $\Lambda \subset T^*L$ be a ridgy Lagrangian. Denote by $R \subset \Lambda$ the ridge locus and let $\Lambda \setminus R = P_1 \cup \cdots \cup P_k$ be the decomposition into connected components (each of which is a smooth manifold with corners). Suppose that $\Lambda$ is graphical over $L$ and denote by $Q_j$ the image of $P_j$ under the projection $T^*L \to L$. Then $\Lambda$ is given over $Q_j$ as the graph of a closed 1-form $\beta_j$. Assume for simplicity that $\Lambda$ is exact, so that we can write $\beta_j=dh_j$ for $h_j:Q_j \to \bR$ a smooth function. Set $\lambda_j = \text{Hess}(h_j)$ on $Q_j$, where we use an auxiliary Riemannian metric on $L$ to write down the Hessian. Note that the $\lambda_j$ assemble to a tectonic field $\lambda$ with plates $Q_j$. 

\begin{definition}
When a tectonic field $\lambda$ arises in this way we say that it is \emph{integrable}.
\end{definition}

A tectonic field provides the infinitesimal data to integrate a graphical ridgy Lagrangian. However, for the integration to be possible in a neighborhood of the fault locus we need the additional condition that the ridges are aligned with the faults. 

\begin{definition}\label{def: integrable}
We say that a tectonic field $\lambda$ is \emph{aligned} if $\tau_j=TN_j$ for every fault $N_j$ and corresponding ridge direction $\tau_j$, see Figure \ref{aligned}.
\end{definition}

  \begin{figure}[h]
\includegraphics[scale=0.6]{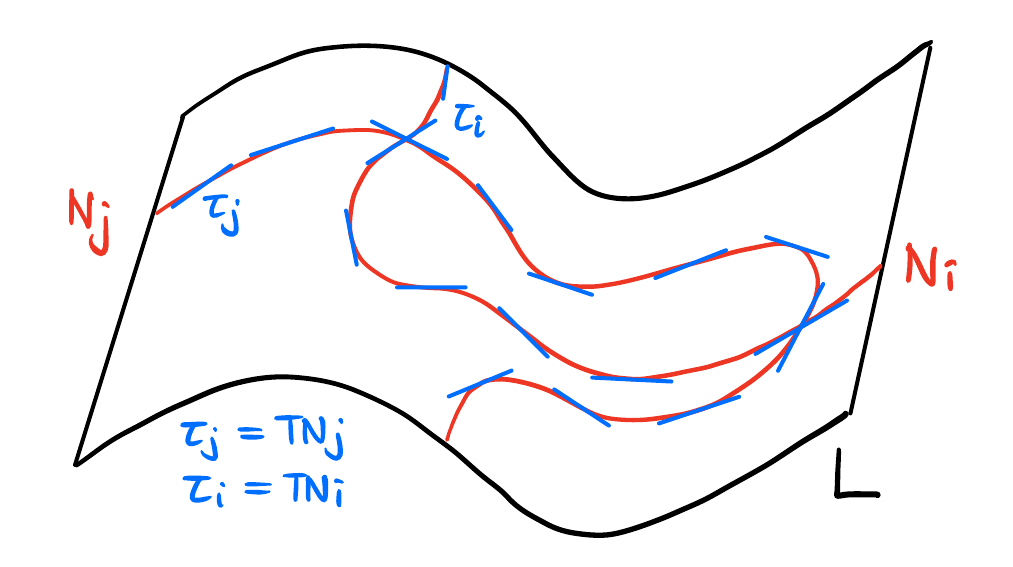}
\caption{An aligned tectonic field.}
\label{aligned}
\end{figure}

%\begin{definition}\label{def: integrable}
%When a tectonic field $\lambda$ arises in this way we say that it is \emph{integrable}.
%\end{definition}

For a Lagrangian plane field $\gamma$ in $T^*L$, the problem under consideration is to deform the zero section $L$ by a ridgy isotopy so that it becomes transverse to $\gamma$. In the previous section we found a formal solution to this transversalization problem, i.e. a tectonic field $\lambda$ such that $\lambda \pitchfork \gamma$. In this section we take a step towards integrability by upgrading our formal solution to an aligned solution. More precisely, we have the following aligned version of Theorem \ref{theorem: formal transversalization}.

%Since any homotopy $\Phi_t$ of linear isomorphisms of $TL$ can be lifted to a homotopy of linear symplectic isomorphisms of $T(T^*L)$ we deduce the following immediate consequence of Theorem \ref{theorem: homotopically aligned formal transversalization}.

\begin{theorem}\label{theorem: aligned formal transversalization}
For any Lagrangian field $\gamma$ there exists an aligned tectonic field $\zeta$ such that $\zeta \pitchfork \widehat \gamma$ for $\widehat \gamma$ a Lagrangian field homotopic to $\gamma$.
\end{theorem}

Note that we gain alignment of the tectonic field $\widehat { \zeta}$ at the cost of deforming the Lagrangian field $\gamma$ to a homotopic field $\widehat{\gamma}$. Nevertheless, in the next section we show that it is possible to integrate the aligned solution $\widehat \zeta$ produced by Theorem \ref{theorem: aligned formal transversalization} to obtain a ridgy Lagrangian which after an ambient Hamiltonian isotopy is transverse to $\gamma$ itself, thus proving our main result Theorem \ref{theorem: main theorem}. 

The rest of the present section is devoted to the proof of Theorem \ref{theorem: aligned formal transversalization}. In fact, we prove below the following more general extension result, which is the aligned analogue of Theorem \ref{theorem: formal transversalization extension}.
 
  \begin{theorem}\label{theorem: aligned formal transversalization extension}
Let $\gamma$ be a Lagrangian field and let $\zeta$ be an aligned tectonic field. For any two disjoint closed subsets $K_1,K_2 \subset L$, there exists an aligned tectonic field $\widehat{\zeta}$ and a Lagrangian field $\wh \gamma$ homotopic to $\gamma$ such that the following properties hold.
\begin{itemize}
\item $\widehat{\zeta}$ is $C^0$-close to $\zeta$.
\item $\widehat{\zeta} \pitchfork  \wh \gamma$ on $Op(K_1)$.
\item $\widehat{\zeta}= \zeta$ on $Op(K_2)$.
\end{itemize}
Moreover, we can assume that the homotopy between $\gamma$ and $\widehat \gamma$ is fixed on $Op(K_2)$.
\end{theorem}

\subsection{Homotopically aligned transversalization}

It will be useful to also consider the homotopical version of definition \ref{def: integrable}.

\begin{definition} 
We say that a tectonic field is \emph{homotopically aligned} if there exists a homotopy of linear isomorphisms $\Psi_t:T_xL \to T_xL$, $x \in L$, such that $\Psi_0=\text{id}_{T_xL}$ and $\Psi_1(\tau_j)=TN_j$.
\end{definition}

We call $\Psi_t$ the \emph{homotopical alignment} and consider it part of the defining data of a homotopically aligned tectonic field. %Note that any homotopy $\Psi_t$ of linear isomorphisms of $TL$ can be lifted to a homotopy of linear symplectic isomorphisms of $T(T^*L)$. Hence in the conclusion of Theorem \ref{theorem: aligned formal transversalization extension} we can absorb this deformation into the homotopy of $\gamma$. 
Note that Theorem \ref{theorem: aligned formal transversalization extension} follows immediately from the following analogous homotopically aligned statement.

 \begin{theorem}\label{theorem: homotopically aligned formal transversalization extension}
Let $\gamma$ be a Lagrangian field and let $\zeta$ be a homotopically aligned tectonic field. For any two disjoint closed subsets $K_1,K_2 \subset L$, there exists a homotopically aligned tectonic field $\widehat{\zeta}$ and a Lagrangian field $\wh \gamma$ homotopic to $\gamma$ such that the following properties hold.
\begin{itemize}
\item $\widehat{\zeta}$ is $C^0$-close to $\zeta$.
\item $\widehat{\zeta} \pitchfork  \wh \gamma$ on $Op(K_1)$.
\item $\widehat{\zeta}= \zeta$ on $Op(K_2)$.
\end{itemize}
Moreover, we can assume that the homotopy between $\gamma$ and $\widehat \gamma$  is fixed on $Op(K_2)$ and that the homotopical alignment $\wh \Psi_t$ for $\wh \zeta$ agrees with the homotopical alignment $\Psi_t$ of $\zeta$ on $Op(K_2)$.
\end{theorem}

\begin{proof}[Proof of Theorem \ref{theorem: aligned formal transversalization extension} assuming Theorem \ref{theorem: homotopically aligned formal transversalization extension}]
We apply Theorem \ref{theorem: homotopically aligned formal transversalization extension} in the case where $\Psi_t=\text{id}_{T_xL}$. The output is $\wh \gamma$ and $\wh \zeta$, with homotopical alignment $\wh \Psi_t$. Let $\Phi_t$ be the unique homotopy of linear symplectic isomorphisms of $T_x(T^*L)$, $x \in L$, lifting the linear isomorphism $\wh \Psi_t$ of $T_xL$ and fixing the cotangent fibre $T_x^*L$. Then taking the aligned tectonic field $\Phi_1(\wh \zeta)$ and concatenating the homotopy between $\gamma$ and $\wh \gamma$ with the homotopy $\Phi_t(\wh \gamma)$ we obtain the conclusion of Theorem \ref{theorem: aligned formal transversalization extension}. \end{proof}

Therefore we have reduced the aligned formal transversalization theorem \ref{theorem: aligned formal transversalization extension} to the homotopically aligned formal transversalization theorem \ref{theorem: homotopically aligned formal transversalization extension}. To prove the homotopically aligned formal transversalization theorem \ref{theorem: homotopically aligned formal transversalization extension} we will take the tectonic field $\wh \zeta$ produced by the formal transversalization theorem \ref{theorem: formal transversalization extension}, which may not be homotopically aligned, and perform a local modification to adjust the homotopical condition obstructing alignment.

\subsection{Formal ridges}

%\subsection{Changing the homotopy class of ridge directions}
 We begin by introducing the notion of a formal ridge.
 
\begin{definition}
A \emph{formal $k$-ridge} over an $n$-dimensional vector space $V$ is the data of a quadratic form $\lambda_0$ on $V$ and an unordered collection of $k$ rank 1 forms $\mu_1 , \ldots , \mu_k$ on $V$ such that each of the hyperplanes $H_j = \ker( \mu_j )$ is transverse to all finite intersections of the other $H_i$, $i \neq j$.
\end{definition}

Let $\lambda$ be a tectonic field on a smooth $n$-dimensional manifold $L$. A point at which exactly $k$ of the faults of $\lambda$ meet determines a formal $k$-ridge. Indeed,  the $2^k$ Lagrangian planes corresponding to the tectonic field $\lambda$ at the point $x$ are given by $\lambda_0 + \sum_{j \in J} \mu_j$, where $J$ ranges over subsets of $\{1,2,\ldots , k\}$ and $\lambda_0$ is the plane corresponding to the quadrant which is initial with respect to the fault co-orientations. We get a formal $k$-ridge by considering $\lambda_0$ together with the $\mu_i$. Note that with this choice of $\lambda_0$ we have  that each $\mu_i$ is the square of a linear form $\ell_j^2$. However, we could also take a different plane in $\lambda$ as our $\lambda_0$ and replace each of the corresponding $\mu_i$ with $-\mu_i$. Then we get another formal $k$-ridge which has the same collection of $2^k$ Lagrangian planes associated to it. Note that there is no canonical ordering on the forms $\mu_i$.

Denote $[\lambda]=\text{span}(\lambda)$, which is a field of coisotropic subspaces of $T(T^*L)|_L$. The dimension of $[\lambda]$ varies and is equal to $n+k$ along the formal $k$-ridge locus. Given a Lagrangian field $\eta$ along $L$ we denote by $\eta^{[\lambda]}$ the symplectic reduction of $\eta \cap [\lambda]$ in $[\lambda]/[\lambda]^{\perp_\omega}$. Note that the transversality of $\eta$ to $\lambda$ is equivalent to transversality of $\eta$ to $[\lambda]$ and transversality of $\eta^{[\lambda]}$ to $ \lambda^{[\lambda]}$. Here $\lambda^{[\lambda]}$ consists of the collection of symplectic reductions of the Lagrangian planes of $\lambda$.
\begin{lemma}\label{lm:contr-fix-red}
 The projection $\eta\mapsto\eta^{[\lambda]}$ defined on the space of Lagrangian fields transverse to $[\lambda]$ has contractible fibers.
  \end{lemma}
  
\begin{proof} Consider the fibre over a formal $k$-ridge point. We factor the projection $\eta \mapsto \eta^{[\lambda]}$ as the map $\eta \mapsto \eta \cap [\lambda]$ and $\eta \cap [\lambda] \mapsto \eta^{[\lambda]}$. The second map is defined on the space of $(n-k)$-dimensional isotropic subspaces of $[\lambda]$. Let $\tau \subset [\lambda]$ be an $(n-k)$-dimensional isotropic subspace. Then the fibre of the second map over the reduction of $\tau$ can be identified with the space of linear maps $\tau \to [\lambda]^{\perp_{\omega}}$, hence is contractible. For the first map, take an $(n-k)$-dimensional isotropic subspace $\tau \subset [\lambda]$ and let $\eta$ be a Lagrangian plane whose intersection with $[\lambda]$ is $\tau$. Then the fibre of the first map over $\tau$ can be identified with the space of quadratic forms on $\eta/\tau$, hence is also contractible.
\end{proof}

For an inductive argument below it will be convenient to consider formal $k$-ridges with a fixed ordering of the rank 1 forms $\mu_j$. We call this an \emph{ordered formal $k$-ridge}.

  \begin{lemma}\label{lm:isom-plus-one}
Let $\lambda^1$ and $\lambda^2$ be two ordered formal $k$-ridges on $V$. There exists a linear symplectic isomorphism $\Phi$ of $V \times V^*$ which sends $\lambda^1$ to $\lambda^2$. Moreover $\Phi$ is determined up to contractible choice by its restriction to $[\lambda']$, where $\lambda'$ is the ordered formal $(k-1)$-ridge obtained from $\lambda^1$ by forgetting $\mu_k$.
%suppose that $\lambda^1_0=\lambda^2_0$ and $\mu^1_i=\mu^2_i$ for $i<k$ and let $\lambda'$ denote the corresponding formal $(k-1)$-ridge common to $\lambda^1$ and $\lambda^2$. Then we can demand that $\Phi$ is the identity on $[\lambda']$, and this determines $\Phi$ up to contractible choice. 
  \end{lemma}
  
  \begin{remark}
  That $\Phi$ sends $\lambda^1$ to $\lambda^2$ means that the image of the Lagrangian plane $\lambda^1_0 + \sum_{j \in J} \mu^1_j $ by $\Phi$ is $\lambda^2_0 + \sum_{j \in J} \mu^2_j $ for every $J \subset \{1 , \ldots , k \}$.
  \end{remark}
  
  %\begin{lemma}\label{lm:isom-plus-one}
 %Let $\lambda$ and $\lambda'$ be two formal $k$-ridges. Then there exists a linear  symplectic isomorphism between $\lambda$ and $\lambda'$. If the corresponding rank 1 forms satisfy $\mu_i=\mu_{i}'$ for $i<k$  then the isomorphism is unique up to contractible choice.
%  \end{lemma}

%\begin{lemma}\label{lm:isom-plus-one}
%Let $\ell_1,\dots, \ell_k$   and  $\ell_1',\dots, \ell_k'$ be two ordered collections of $k$ linear quadratic forms in general position, where $k \leq \dim(V)$. Then there exists a linear  symplectic isomorphism between $\lambda$ and $\lambda'$. If $\mu'_i=\mu_{i}$ for $i<k$  then the isomorphism is unique up to contractible choice.
% \end{lemma}
  
\begin{proof} We argue by induction on $k=0,1,\ldots , n$. For $k=0$ the existence part follows from the fact that the symplectic group acts transitively on the Lagrangian Grasmannian. The uniqueness follows from the fact that a linear symplectic isomorphism is determined by its restriction to a pair of transverse Lagrangian planes, together with the fact that the space of Lagrangian planes transverse to a fixed Lagrangian plane is contractible. We spell out the details of an explicit argument which will be easily adaptable to the case $k>0$. First we reduce to the case $V=\bR^n$, $\lambda^1_0=\lambda^2_0=\bR^n\subset \bC^n$ and $\Phi|_{\bR^n}=\text{id}_{\bR^n}$. Write the symplectic matrix $M \in \text{Sp}(2n)$ representing $\Phi$ in the block form corresponding to $\bC^n=\bR^n \times i \, \bR^n$
\[ M =\mat{A & B\\
C&D }. \]
Then $\Phi|_{\bR^n}=\text{id}_{\bR^n}$ is equivalent to $A=I_n$ and $C=0$. That $M$ is symplectic means $M^T \Omega M = \Omega$ for
\[ \Omega =\mat{0 & I_n\\
-I_n& 0 },\]
where $I_n$ is the $n$ by $n$ identity matrix. It follows that $D=I_n$ and $B^T=B$. Hence $\Phi$ is uniquely determined up to the contractible choice of the symmetric matrix $B$. This completes the base case.

For the inductive step, observe that as before it suffices to consider the case $\lambda^1_0=\bR^n \subset \bC^n$. Furthermore, up a linear change of coordinates in $\bR^n$ we may assume that the kernel of both $\mu^1_j$ and $\mu^2_j$ is the coordinate hyperplane $\{q_j=0\} \subset \bR^n$. Let $\Phi$ be the linear symplectic isomorphism obtained by applying the inductive hypothesis to the 
%formal $(k-1)$-ridges $\lambda_{\text{old}}$ and $\lambda'_{\text{old}}$ corresponding to $\mu_1 , \ldots , \mu_{k-1}$ and $\mu_1' , \ldots \mu_{k-1}'$. Denote by $\lambda_{\text{new}}$ and $\lambda'_{\text{new}}$ the formal $k$-ridges obtained when $\mu_k$ and $\mu'_k$ are added to the collection. 
ordered formal $(k-1)$-ridges corresponding to $\lambda^1$ and $\lambda^2$ after forgetting $\mu_k^1$ and $\mu_k^2$ respectively. Then by pulling $\lambda^2$ back by $\Phi$ we reduce to the case $\lambda^1_0=\lambda^2_0$ and $\mu_i^1=\mu_i^2$ for $i<k$.

%Let $\lambda^j_{\text{old}}$ be the formal $(k-1)$-ridge obtained by forgetting $\mu_k^j$, $j=1,2$, and llet $\lambda_{\text{new}}$ be the formal $k$-ridge corresponding to $\mu_1 , \ldots , \mu_k$  and let $\lambda_{\text{new}}'$ be the pullback by $\Phi$ of the formal $k$-ridge corresponding to $\mu_1' , \ldots , \mu_k'$. %Note that the jumps $\mu_i''$ of $\lambda_{\text{new}}'$ satisfy $\mu_i''=\mu_i$ for $i<k$. 

In this case have $[\lambda']=\{ p_j = 0 , \, j > k-1 \}$ and $[\lambda^1]=[\lambda^2]=\{ p_j = 0 , \, j > k \}$. %By the inductive hypothesis $\lambda_{\text{new}}$ corresponds to forms $\mu_1'' , \ldots \mu_k''$ such that $\mu_j=\mu_j''$ for $j<k$. 
Note that the product of a horizontal shear of the symplectic subspace $(q_k,p_k)$ and the identity on the complementary $\bR^{2n-2}$ fixes $[\lambda' ]$. Since the group of horizontal shears $(x,y) \mapsto (x, y + ax)$, $a \in \bR$, acts transitively on the space of lines in $\bR^2$ transverse to the horizontal axis $\{y = 0 \} \subset \bR^2$, we can find a linear symplectic isomorphism which is the identity on $[\lambda']$ and takes $\lambda^1$ to $\lambda^2$. This proves the existence part.

For the uniqueness part it suffices to show that a linear symplectic isomorphism $\Phi$ of $\bR^{2n}$ which restricts to the identity on $[\lambda^1]=\{ p_j=0 ,  \, j \geq k \}$ is unique up to contractible choice. With the same notation as above, write the symmetric matrix $B$ in block form
\[ B = \mat{X & Y \\
Y^T & W }. \]
Here $X$ is a $k$ by $k$ matrix and $W$ is an $(n-k)$ by $(n-k)$ matrix, which are both symmetric. The conditions on $\Phi$ are equivalent to $X=0$ and $Y=0$
%and that the only nonzero entry of $X$ can be the bottom right one $x_{kk}$, which is determined by the image of $\mu_k$ under $\Phi$.
%Similarly, the only nonzero row of $Y$ can be the last one (last column of $Y^T$). 
%implies that $X=0$. For example, consider the Lagrangian plane $\lambda_1$ which is given by a rank 1 quadratic form $\alpha_1 q_1^2$, $\alpha_1 \neq 0$. Since $\Phi$ fixes $\lambda_1$, we deduce that the $(n+1)$st column of $M$ must have a zero in its first entry, i.e the top left entry $x_{11}$ of $X$ is zero. By considering $\alpha_2 q_2^2$ we similarly deduce that $x_{22}=0$. By considering $\alpha_1 q_1^2 + \alpha_2 q_2^2$ we deduce that $x_{12}=x_{21}=0$, etc. 
%Conversely, if those conditions on $X$ and $Y$ are met then we see that $\Phi$ fixes $\{ p_j = 0 , \, j \geq k -1 \}$ and preserves $\{ p_j = 0 , \, j \geq k \}$. 
Hence $\Phi$ is uniquely determined up to the contractible choice of the symmetric matrix $W$. \end{proof}  
%Since $H_1 \cap \cdots\cap  H_{k-1}$ is transverse to $H_k$, we have that $H_1 \cap \cdots \cap H_{k-1} \cap H_k$ is a hyperplane in $H_1 \cap \cdots \cap H_{k-1} $. Let $P$ be a symplectic subspace which intersects $H_1 \cap \cdots \cap H_{k-1}$ in a line transverse to $H_1 \cap \cdots \cap H_{k-1} \cap H_k$ and $Q$ a complementary symplectic subspace containing $\lambda_{\text{old}}$. 

%Note that $\lambda_{\text{old}}$ has codimension $1$ in $\lambda_{\text{new}}$. So we can pick a symplectic subspace $P \subset V$ containing $\lambda_{\text{old}}$ and a complementary symplectic subspace $Q \subset V$ whose intersection with $\lambda_{\text{new}}$ is a line transverse to $\lambda_{\text{old}}$. We can then rotate 

%et $Q \subset V$ be a symplectic subspace containing $\lambda_{\text{old}}$ and $P \subset V$ a complementary symplectic subspace which intersects 

%So we can find a nonzero linear function whose kernel contains $H_1 \cap \cdots \cap H_{k-1}$ but does not vanish on $H_k$. 

%Let $V \simeq \bR^n$ be an isomorphism which takes $H_j$ to $\{q_j=0\}$. The subgroup of $U$
 
  \subsection{The model}\label{section: the model}

Consider a tectonic field $\lambda$ on $L$. Let $N \subset L$ be one of its faults, which bounds a domain $U$ such that outside of $U$ the field $\lambda$ differs by adding a rank one 1 quadratic form $\mu$ along $N$. Let $\Omega \subset U$ be a domain with boundary and corners, where we decompose $\p_1\Omega=F_1\cup F_2$ for $F_1$ and $F_2$ smooth so that $F_1= \p \Omega\cap N$, $\partial F_2 = F_2 \cap N$ and the corner is precisely $\p_2\Omega= F_1\cap F_2$. Let $\nu$ be a field of rank 1 quadratic forms on $\Omega$ such that $\nu=\mu$ near $F_1$.  Consider the field $\wh\lambda$ which is defined to be $\lambda$ outside of $\Omega$ and $\lambda + \nu$ on $\Omega$. After smoothing, $\wh\lambda$ becomes a tectonic field with $\wh N = (N \setminus F_1)\cup F_2$ as one of its faults, see Figure \ref{modification}.

  \begin{figure}[h]
\includegraphics[scale=0.65]{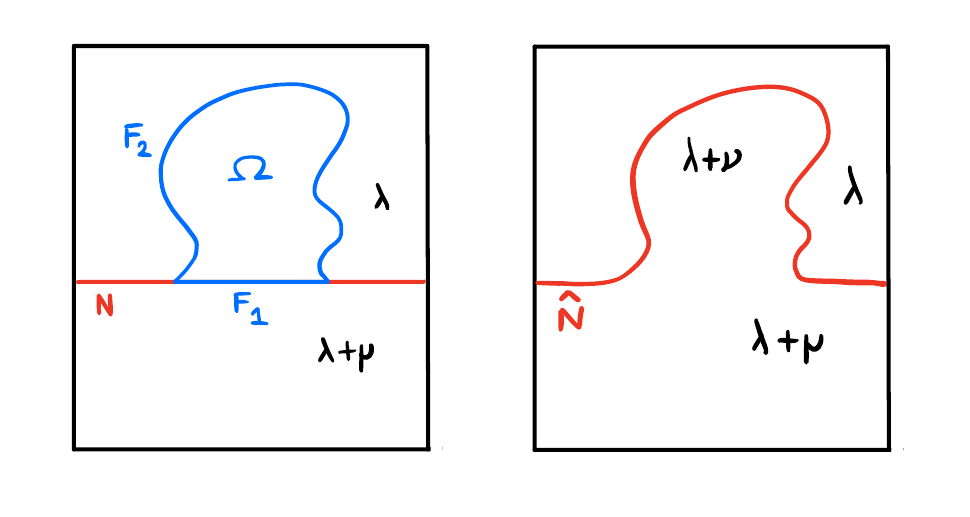}
\caption{The modification $\lambda \mapsto \wh \lambda$. }
\label{modification}
\end{figure}

We now apply this construction in a specific model. Consider a tectonic field  $\lambda$ over $\bR^n\subset T^*\bR^n$  with faults along the first $k$ coordinate hyperplanes $Q^j:=\{q_j=0\}$, $j=1,\dots, k \leq n$. We use the notation $Q^J$ for the fault  intersections: $ Q^J=\bigcap_{j\in J} Q^j$, $J \subset K$ for $K= \{ 1 , \ldots , k \}$. We also enumerate the quadrants on $\bR^n$ by multi-indices $I \subset K$, namely $C_I = \{ q_i \geq 0 , \, i \in I ; \, \, q_j \leq 0, \, j \in K \setminus I \}$. We assume that $\lambda$ is constant in the sense that the discontinuities of $\lambda$ across a fault $Q^j$ are given by constant rank 1 quadratic forms $\mu_j$. So we may write $\lambda= \lambda_0 + \sum_{j \in J} \mu_j $ on $C_J$ for $\lambda_0$ a fixed Lagrangian plane.
 
Take a sphere $\Sigma\subset\bR^n$ of radius  $1$ centered at a point $a$ with coordinates $q_k=2$, $q_j=0,j\neq k$. Denote  $A=\{0\leq q_k\leq 1; q_j=0, j\neq k \}$, and denote by $\Omega$ a neighborhood of $A\cup\Sigma$ in $\{q_k\geq 0\}$. Thus $\p\Omega= (F_1\cup F_2)\cup F_3$, where $F_1= \p\Omega\cap Q^k$, $F_2$ is a $(n-1)$-disk transverse to $Q^k$ and $F_3$ is a $(n-1)$-sphere disjoint from $Q_k$. Let $\nu$ be a field of rank 1 quadratic forms over $\Omega$ which agrees with $\mu_k$ over $F$. We will additionally assume that near $Q^J$ the field  $\nu$ is independent of coordinates $q_j$, $j\in J$. Performing the above construction to $\lambda$ for the specific choices of $\Omega$ and $\nu$ yields a tectonic field which we denote $\wh \lambda$, see Figure \ref{3Dcase}.

  \begin{figure}[h]
\includegraphics[scale=0.6]{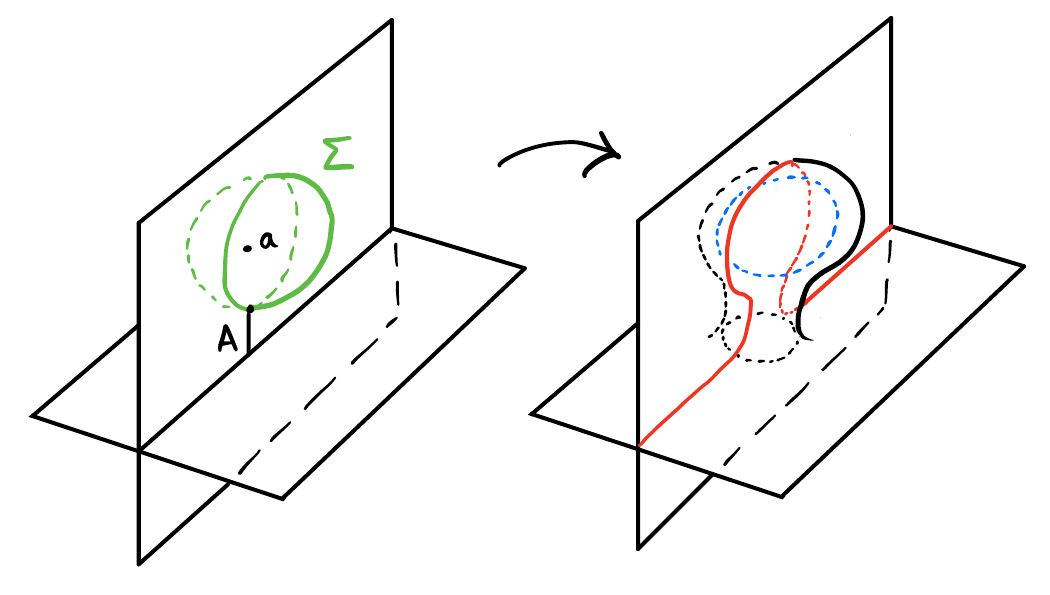}
\caption{The standard model in the case $n=3$, $k=2$. Note that the vertical fault does not change in the process, only the horizontal fault changes.}
\label{3Dcase}
\end{figure}

 %This completes the description of the model in $\bR^n$ at a $k$-multiple fault point where $k<n$. When $k=n$ we write the decomposition of $\lambda$ as a formal ridge $\lambda = \lambda_0+ \sum_{j=1}^n \mu_i$ (so $\lambda_0$ is a Lagrangian plane and $\mu_i$ is a tectonic field on $\R^n$ with fault $\{q_i=0\}$). Set $\lambda' = \lambda_0 + \sum_{i=1}^{n-1}\mu_i$. Then the origin is an $(n-1)$-multiple fault point, so for the specified data $\Omega$ and $\nu$ we can implant the above model to obtain $\wh \lambda'$. We then set $\wh \lambda = \wh \lambda' + \mu_n$, which has faults the union of the faults of $\wh \lambda'$ with $\{q_n=0\}$. Compare Figures \ref{fig: obstructionresolution} and \ref{2Dcase} below.

  \begin{proposition}\label{prop:model-change} Let $\gamma$ be a constant Lagrangian distribution transverse to $\lambda$. Then $\wh \lambda$ is transverse to a distribution $\wh\gamma$ which is homotopic to $\gamma$ by a deformation fixed outside of a compact set.
 \end{proposition}
 \begin{proof}
%We enumerate quadrants on $\bR^n$   by multi-indices $I=\{i_1,\dots, i_m\}\subset K:=\{1,\dots, k\}$, $C_I:=\{q_i\geq 0, i\in I, q_j\leq0, j\in K\setminus I\}$.  In particular, $C_K$ is  the quadrant with all positive coordinates (with indices $\leq k$). 
 %We use the notation $Q^J$ for the fault  intersections: $ Q^J=\bigcap\limits_{j\in J} Q^j$. 
% We write  $C^J_I:=C_I\cap Q_J$.
Write $C^J_I = C_I \cap Q^J$. Note that on $C_I^J$ the tectonic field $\lambda$ is a fixed formal $r$-ridge $\lambda^J_I$, where $r\leq k$ is the cardinality of $J$. We construct $\wh \gamma$ inductively over the dimension $n-|J|$   of  the strata $C^J_I$
intersecting $\Sigma$. The smallest dimensional stratum is $C=C_{1, \ldots , k}^{1,\ldots , k-1}$. For every point $x\in C \cap\Omega$ consider a linear symplectic isomorphism $\Phi_{ 0,x}$ of $\bR^{2n}$ which sends $\wh\lambda_0=\lambda_0$ to $\wh\lambda_x$ and is the identity on $[\lambda']$. Here  $0$ denotes the origin in $\bR^n$, $\lambda_x$ is the formal $k$-ridge of $\lambda$ at $x$ and $\lambda'$ is the formal $(k-1)$-ridge obtained from $\lambda_0$ by forgetting $\mu_k$. According to Lemma \ref{lm:isom-plus-one} there exists a homotopically unique continuous family of such isomorphisms. We can assume that $\Phi_{0,x}$ is the identity if $\wh\lambda_x=\wh\lambda_0$. Let us define $\wh\gamma_x=\Phi_{0,x}(\gamma_0)$, which is transverse to $\wh \lambda$. Using Lemma \ref{lm:contr-fix-red} we can extend $\wh\gamma$ to  $C$ keeping fixed its reduction $\wh\gamma^{[\lambda']}=\gamma^{[\lambda']}$ and making it equal $\gamma$ outside a neighborhood of $\Omega$. Again applying Lemma \ref{lm:contr-fix-red} we conclude that  the constructed field $\wh\gamma$ is homotopic to $\gamma$ via a homotopy $\gamma_t$ with a fixed reduction $\gamma_t^{[\lambda']}$. 

%For every point $x\in C_K^{\{1,\dots, k-1\}}\cap\Omega$ consider a linear isomorphism $L_{a,x}$ between $\wh\lambda(a)$ and $\wh\lambda(x)$. According to Lemma \ref{lm:isom-plus-one} there exists a homotopically unique continuous family of such isomorphisms. We can assume that $L_{a,x}$ is identity if $\wh\lambda(x)=\wh\lambda(a)$. Let us define $\wh\eta(x):=L_{a,x}(\wh\eta(a)).$ Using Lemma \ref{lm:contr-fix-red} we can extend $\wh\eta$ to  $C_K^{\{1,\dots, k-1\}}$ keeping fixed its reduction $\wh\eta^{[\lambda]}$, and making it equal $\eta$ outside a neighborhood of $\Omega$. Again applying Lemma \ref{lm:contr-fix-red} we conclude that   the constructed field $\wh\eta$ is homotopic to $\eta$ via a homotopy
%$\eta_t$ with a fixed reduction $\eta_t^{[\lambda]}$. 

Next, we extend  $\wh\gamma$ to a neighborhood of $C=C_{\{1, \ldots , k\}}^{\{1,\dots, k-1\}}$ so that it is independent of the coordinates $q_j, j\leq k-1$.
For any stratum  $C_I^J$ of codimension $k-2$ adjacent  to 
 $C$  we choose a point $y \in C_I^J\cap\Omega$ in a neighborhood  $U$ of $C$ where $\wh\gamma$ is already defined. We note that  in this neighborhood there exists a family of  linear isomorphisms $\Phi_{y,x}$ fixing $[\lambda^I_J]$ which maps $\wh\lambda_y$ to $\wh\lambda_x$ and  $\wh\gamma_y$ to $\wh\gamma_x$, $x\in C_I^J\cap\Omega\cap U$. We extend the family to all $x\in C_I^J\cap\Omega$ and define $\wh\gamma_x:=\Phi_{y,x}(\wh\gamma_y)$. Next we extend it to $C_I^J$ keeping fixed its reduction $\wh\gamma^{[\lambda^I_J]}$, and making it equal $\gamma$ outside a neighborhood of $\Omega$. The same lemma implies that the constructed field $\wh\gamma$ is homotopic to $\gamma$ via a homotopy $\gamma_t$ with a fixed reduction $\gamma_t^{[\lambda^I_J]}$. Continuing this process we construct the required distribution $\wh\gamma$.   \end{proof}

 \subsection{Changing the homotopy class of the ridge directions}\label{section: changing} Finally we show how the local model constructed above can be used to prove Theorem \ref{theorem: homotopically aligned formal transversalization extension}. 

  \begin{proposition}\label{prop:modifying-hom}
  Let $\lambda$ be a tectonic field over a manifold $L$ which is  transverse to a Lagrangian distribution $\gamma$. Then there exists a homopically aligned tectonic field $\wh\lambda$ which is transverse to a Lagrangian distribution $\wh\gamma$ homotopic to $\gamma$. 
    \end{proposition}
    \begin{proof} Let $N_j$ denote the faults of $\lambda$. We will align the ridge directions inductively over the strata of the fault locus $\bigcup_j N_j$. In fact, we will align the outwards normal to $N_j = \partial \Omega_j$, where we recall $\Omega_j \subset L$ is a domain, and where an arbitrary Riemannian metric on $L$ is understood. 
    
    We begin by refining the stratification to a triangulation $\bigcup_j \Delta_j$, so that $\bigcup_{j} N_j$ is contained in the $(n-1)$-skeleton and moreover such the interior of each $k$-simplex is entirely contained in the locus where exactly $r$ of the faults $N_j$ intersect for some $r \leq n-k$ (which depends on the $k$-simplex).   
    
Choose an order of the faults $L_1, \ldots , L_m$. We recall that $N_j = \partial \Omega_j$ for some domain $\Omega_j \subset L$. Let $J \subset \{1, \ldots , m\}$ and consider a point $x \in \bigcap_{j \in J}  N_j\setminus \bigcup_{j \notin J}N_j$. Set $r=|J|$. We have two $r$-frames at $T_xL$. One is given by the $r$-tuple of outward normals to the domains $\Omega_j$, $j \in J$, with the induced order from $\{1,\ldots, m\}$. The other is given by the ridge directions of the tectonic field $\lambda$, with the same order. The co-orientation of the ridge directions (which are hyperplane fields) is specified by a choice of non-vanishing 1-forms $\ell_j$ on $TL|_{N_j}$ which square to the rank 1 quadratic form giving the jump of $\lambda$ over $N_j$. There are two choices for each $N_j$, but either will do. 

We begin our homotopical alignment along the 0-skeleton of the triangulation. At a point where $r \leq n$ of the faults meet we have two elements of $V_{r,n}$, the Stiefel manifold of $r$-frames in $\bR^n$, which is connected for $r<n$. So for points where $r<n$ we may define the homotopical alignment $\Psi_t$ in an arbitrary way, but for points where $r=n$ we may only align $n-1$ of the normals: there is a $\pi_0 V_{n,n}=\bZ/2$ obstruction to aligning the last one, namely the orientation of the frame.

We need to modify $\lambda$ near the points of the 0-skeleton where the obstruction is nontrivial. To do this, let $x \in L$ be such a point and choose $\nu=c\ell^2$ in Proposition \ref{prop:model-change} in such a way that the hyperplane $\tau=\{\ell=0\}$ is tangent to $\Sigma$ with the center of the sphere on an $(n-1)$-multiple fault point adjacent to $x$. This  removes the point $x$ from the 0-skeleton and creates three new $n$-multiple fault points, two of them on the new spherical fault. These last two $n$-multiple fault points can be arranged to have trivial $\bZ/2$ obstruction if we agree that the new spherical fault is the boundary of the domain $\Omega_0 \subset L$ formed by the $n$-ball it bounds and we agree to place this new fault first in our ordering of faults. The other new $n$-multiple point has trivial $\bZ/2$ obstruction by construction, since the relevant ridge direction has been modified by a half-turn, see Figure \ref{fig: obstructionresolution}. Note that according Proposition \ref{prop:model-change}, the new tectonic field is transverse to a Lagrangian distribution homotopic to $\gamma$.
%Begin by defining the homotopical alignment $\Psi_t$ in an arbitrary way at the 0-simplices. 
      \begin{figure}[h]
\includegraphics[scale=0.65]{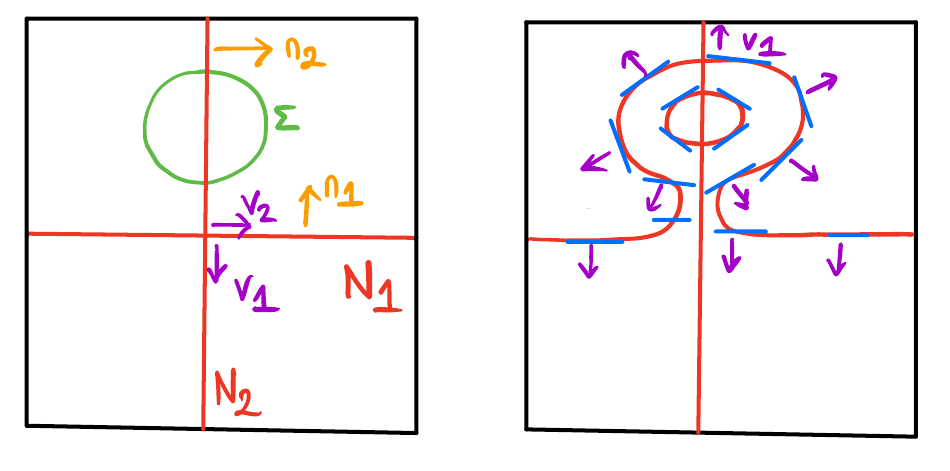}
\caption{The local modification for the obstruction $\pi_0 V_{n,n}=\bZ/2$ in the case $n=2$. In this example $(n_1,n_2)$ is the frame corresponding to the ordered pair of faults $N_1,N_2$ and $(v_1,v_2)$ is the frame corresponding to the ridge directions. We assume that $n_2$ and $v_2$ have been lined up but $n_1$ and $v_1$ differ by $-1$, so the orientations determined by the frames don't agree. The model changes the sign of $v_1$ at the order 2 ridge where $N_2$ intersects the new $N_1$, as well as creating two new order 2 ridges along the intersection of $N_2$ with the new spherical fault for which the $\bZ/2$ obstruction is trivial.}
\label{fig: obstructionresolution}
\end{figure}

We have achieved alignment of the frames on the 0-skeleton, the base case of our inductive argument. The inductive step is similar. Indeed, suppose that the homotopical alignment is defined on the $(k-1)$-skeleton of $L$ and let $C$ be a $k$-simplex in $\Delta_k$ where exactly $r \leq n-k$ of the faults intersect. We must extend the homotopical alignment from $\partial C \simeq S^{k-1}$ to $C \simeq D^k$. % By considering the vectors normal to the ridge directions we keep thinking in terms of the Stiefel manifold of $r$-frames $V_{r,n}$ in $\bR^n$. 
Since $\pi_{k} V_{r,n} =0$ for $k<n-r$ we can always homotopically align the ridge directions if $r< n-k$, and if $r=n-k$ we can align all but one. 

The obstruction to aligning that last ridge direction lies in $\pi_{k} V_{n-k,n} $, which is $\bZ$ if $n-k$ is even or $k=1$ and $\bZ/2$ if $n-k$ is odd and $k>1$. We claim that one can change this obstruction by $\pm 1$ by applying Proposition \ref{prop:model-change}. To see this consider the fibration
\[ S^k \to  V_{n-k,n} \to V_{n-k-1,n} \]
and the following portion of its long exact sequence in homotopy
\[ \bZ \simeq \pi_k S^k \to \pi_k V_{n-k,n} \to \pi_k V_{n-k-1,n} \simeq 0. \]

We deduce that one can realize the generator of $\pi_k V_{n-k,n}$ as the map $S^k \to V_{n-k,n}$ which fixes $n-k-1$ elements of the frame and lets the last element trace out a $k$-sphere in the complementary $k+1$ dimensional space. 

Returning to the proof, choose $\nu=c\ell^2$ in Proposition \ref{prop:model-change} in such a way that the hyperplane $\tau=\{\ell=0\}$ is tangent to $\Sigma$ with the center of the sphere on a $(k-1)$-multiple fault component $C'$ adjacent to $C$. Then we create a new spherical fault. From the above description of a generator for the cyclic group $\pi_k V_{n-k,n}$ it follows that performing this operation changes the homotopy class of the ridge field on the component $C$ by $\pm 1$ depending on the choice of the component $C'$, see Figure \ref{2Dcase}.
    
      \begin{figure}[h]
\includegraphics[scale=0.65]{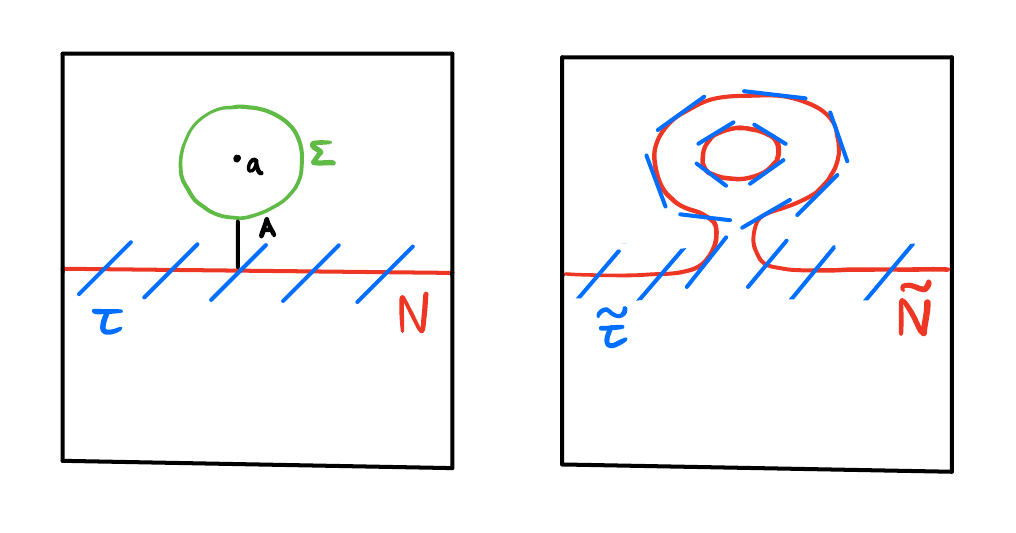}
\caption{The local modification in the case $n=2$, $k=1$. Note that the homotopy class of the line field changes by $\pm 1$ relative to the endpoints.}
\label{2Dcase}
\end{figure}
    
%  \begin{figure}[h]
%\includegraphics[scale=0.6]{changingclass}
%\caption{Homotopical alignment of the ridge direction relative to the endpoints using the local modification.}
%\label{changingclass}
%\end{figure}

Hence, we can inductively adjust the homotopy classes of ridges along fault components of decreasing multiplicity until we get a homotopically aligned field $\widehat \lambda$ which according to Proposition \ref{prop:model-change} is transverse to a Lagrangian distribution homotopic to $\gamma$.
 \end{proof}
 
 \begin{remark} From the proof we see that Proposition \ref{prop:modifying-hom} holds in relative form. This means that if there exists a homotopy of linear isomorphisms $\Psi_t : T_xL \to T_xL$ such that $\Psi_0 = \text{id}_{T_xL}$ and $\Psi_1(\tau_j)=TN_j$ on $Op(A)$ for $A \subset L$ closed, then we can demand that $\widehat \lambda = \lambda$ on $Op(A)$ and that the homotopical alignment of $\wh \lambda$ agrees with $\Psi_t$ on $Op(A)$. Moreover, we can demand that the homotopy of $\gamma$ is constant on $Op(A)$. \end{remark}

\begin{proof}[Proof of Theorem \ref{theorem: homotopically aligned formal transversalization extension}] Consider the tectonic field $\wh \zeta$ which is produced by the formal transversalization Theorem \ref{theorem: formal transversalization extension}. Then by applying the relative form of Proposition \ref{prop:modifying-hom} to $\wh \zeta$ and $\gamma$ with $A=K_2$ we obtain the desired homotopically aligned tectonic field. \end{proof}

\section{Integrable solution}\label{section: integrable solution}
\subsection{Holonomic approximation of ridges}\label{section: holonomic}

We now turn to the proof of our main theorem \ref{theorem: main theorem}. Our first task will be to solve the transversalization problem near the ridge locus, where the homotopical information is concentrated. Since the ridge locus is a stratified subset of codimension $1$, we can apply the method of holonomic approximation.  %In the next section we will deal with the complement using the h-principle for the simplification of caustics, which involves a wrinkling technique.
%It will be convenient to use the notion of a tangential rotation. Recall that the Gauss map  $G_\Lambda$ of a Lagrangian submanifold $\Lambda$ of a symplectic manifold $(M, \omega)$ is the map which to each $q \in L$ assigns the linear Lagrangian subspace $G_\Lambda(q)=T_q \Lambda \subset T_qM$.

%\begin{definition}
%A tangential rotation of a Lagrangian submanifold $\Lambda$ of a symplectic manifold $(M, \omega)$ is a homotopy of Lagrangian plane fields $G_t(q) \subset T_qM$ along $L$, starting at the Gauss map $G_0=G_\Lambda$.
%\end{definition}
%\begin{remark} 
%We always assume that $G_t$ is constant, i.e. $G_t=G_\Lambda$, outside of a compact set.
%\end{remark}

%The above definition makes sense even when $\Lambda$ is singular, for example ridgy, though in this case $G_\Lambda$ is only piecewise continuous. However, in this case we will always assume that the tangential rotation $G_t$ is constant in a neighborhood of the singular locus. The goal of the present section is to establish the following proposition.

\begin{proposition}\label{prop: near ridges}
Let $\gamma$ be a Lagrangian field on $T^*L$. There exists a $C^0$-small ridgy isotopy $L_t \subset T^*L$ of $L_0=L$ such that $L_1 \pitchfork \gamma$ in a neighborhood $Op(R)$ of the ridge locus $R \subset L_1$ and such that there exists a Lagrangian field $\widetilde \gamma$ homotopic to $\gamma$ by a homotopy fixed on $Op(R)$ satisfying $L_1 \pitchfork \widetilde \gamma$ %tangential rotation $G_t$ of $L_1$, fixed in that same neighborhood, for which $G_1 \pitchfork \gamma$ 
everywhere.
\end{proposition}

\begin{remark}
The relative version is as follows: if $L \pitchfork \gamma$ on $Op(A)$ for $A \subset L$ a closed subset, then we can demand that $L_t=L$ on $Op(A)$, that $\widetilde \gamma= \gamma$ on $Op(A)$ and furthermore that the homotopy is fixed on $Op(A)$.
\end{remark}

     \begin{figure}[h]
\includegraphics[scale=0.65]{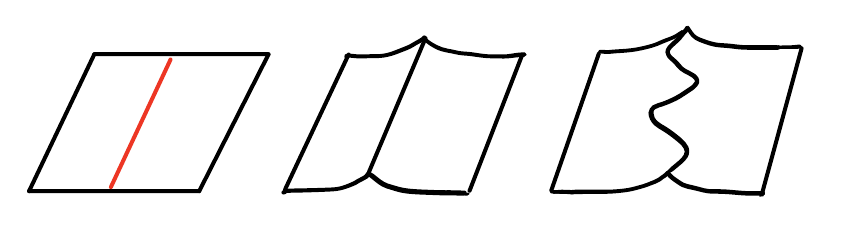}
\caption{Using holonomic approximation to wiggle the ridges.}
\label{wigglingridges}
\end{figure}

As a first step towards Proposition \ref{prop: near ridges} we have the following lemma.

\begin{lemma}\label{lemma: model}
There exists a $C^0$-small integrable tectonic field $\wt \zeta$ on $L$ which is transverse to a Lagrangian field $\wt \gamma$ homotopic to $\gamma$.
\end{lemma}

\begin{proof}
We begin by invoking Theorem \ref{theorem: aligned formal transversalization}, which produces an aligned tectonic field $\zeta$ on $L$ such that $\zeta \pitchfork \wh \gamma$ for $\wh \gamma$ a Lagrangian distribution homotopic to $\gamma$.
%uch that $\zeta \pitchfork \gamma$. Let $\psi_t:TL \to TL$ be as in the definition of homotopic integrability \ref{def: homotopically integrable}, so that $\psi_1(\tau_j)=TN_j$. Pick an almost complex structure $J$ on $T^*L$ which is compatible with the symplectic form and such that $J(T_qL)=T^*_qL$ for $q \in L$. Denote by $\Psi_t: T(T^*L)|_L \to T(T^*L)|_L$ the lift of $\psi_t$ to the symplectic bundle isomorphism $\Psi(v + Jw)=\psi(v)+J\psi(w)$, $v,w \in TL$. Set $\zeta_t=\Psi_t(\zeta)$ and $  \gamma_t = \Psi_t(\gamma)$, so that $\zeta_1 \pitchfork \gamma_1$. 
%If $P$ is given by $p=Aq$ then $\Psi_1(P)$ is given by  $p = \psi A \psi^{-1} q$.
%Suppose that $\zeta^+ - \zeta^- = \eta_j$ along one of the faults $N_j$ of $\zeta$.
%If we think of $\zeta$ and $\zeta_1$ as a families of quadratic forms $T_qL \to \bR$ we have $\zeta_1 = \psi_1^*\zeta =  \zeta \circ \psi_1^{-1}$. Hence the ridge directions of $\zeta_1$ are tangent to the faults $N_j$. 
The tectonic field $\zeta$ jumps discontinuously along a fault $N_j$ by a family of rank 1 forms $\mu_j$. Since $\ker(\mu_j)=TN_j$, we can write $\mu_j(x)=f_j(x) du_j^2$ where $(x,u_j) \in N_j \times (-\ve, \ve)$ are tubular neighborhood coordinates for $N_j = N_j \times 0$ in  $L$ and $f_j :N_j\to \bR$ is a nowhere vanishing smooth function. By reversing the orientation of $(-\ve, \ve)$ if necessary we may assume that $\zeta_1^+ = \zeta_1^- + \mu_j$, where $\zeta_1^{+}$ and $\zeta_1^{-}$ are the extensions of $\zeta_1|_{ \{u_j<0\} }$ and $\zeta_1|_{\{u_j>0\}}$ to $N_j$ respectively. Consider the function 
\[ h_j(x,u_j) =  \frac{1}{2}\psi(u_j) f_j(x)u_j^2 , \]
where $\psi: (-\ve,\ve) \to [0,1]$ is a cutoff function such that $\psi=1$ near $0$ and $\psi=0$ near $\pm \ve$. Hence along $N_j$ we have $\text{Hess}(h_j) = \mu_j$. Note that $h_j$ is compactly supported in the tubular neighborhood $N_j \times (-\ve,\ve)$.  %(note that $dh_j|_{N_j}=0$ so $\text{Hess}(h_j)$ is well-defined along $N_j$ without needing an auxiliary Riemannian metric). 
Consider next the function

\[r_j(x,u_j)=
 \begin{cases} 
    \, \, \,  \,  \frac{1}{2}h_j(x,u_j)  &u_j \geq 0, \\
    
    -  \frac{1}{2}h_j(x,u_j) & u_j \leq0 ,
   \end{cases}
\]
\\
which also has compact support in the tubular neighborhood $N_j \times (-\ve,\ve)$. Set $r= \sum_j r_j : L \to \bR$, a piecewise $C^2$ function which (after choosing an auxiliarty Riemannian metric) generates an integrable tectonic field $\wt \zeta$. Note that if $\zeta$ was $C^0$-small then $\wt \zeta$ is also $C^0$-small. Moreover, we claim that $\wt \zeta \pitchfork \wt\gamma$ for $\widetilde \gamma$ a Lagrangian distribution homotopic to $\gamma$.

Consider the tectonic field $\lambda =   \zeta -  \widetilde \zeta $. Mote that it is a continuous Lagrangian field, because the discontinuities of $\zeta$ are exactly canceled by those of $\widetilde \zeta$. Observe that $\widetilde \zeta$ is graphical, hence $\det(\widetilde \zeta)$ is bounded. Let $\Omega_C=\{  |\det( \wh \gamma )| < C \}$. It follows that for $C>0$ large enough, we have $\widetilde \zeta \pitchfork \wh \gamma$ outside of $\Omega_C$. Moreover for $C>0$ large enough $\wh \gamma$ is homotopic to a Lagrangian plane field $\wt \gamma$ which is equal to $\wh \gamma-\lambda$ on $\Omega_C$ and transverse to $\widetilde \zeta$ outside of $\Omega_C$. Note that the expression $\wh \gamma - \lambda$ makes sense on $\Omega_C$ because $\wh \gamma$ is graphical on $\Omega_C$. 
The claim, and therefore also the Lemma, now follow. Indeed, the condition $\widetilde \zeta \pitchfork \widetilde \gamma$ on $\Omega_C$ is equivalent to the nonsingularity of the form $\widetilde \gamma - \widetilde \zeta = (\wh \gamma-\lambda)-(\zeta - \lambda ) = \wh \gamma-\zeta$, which is in turn equivalent to $\zeta \pitchfork \wh \gamma$, which is true. \end{proof}

\begin{remark}
In the relative version where $L \pitchfork \gamma$ and $\zeta=0$ on $Op(A)$ for $A \subset L$ a closed set, we demand that the homotopy between $\gamma$ and $\widetilde \gamma$ is constant on $Op(A)$.
\end{remark}

%Consider the tectonic field $\lambda = \zeta_1 - \widetilde \zeta$. It is in fact a smooth Lagrangian field, because the discontinuities of $\zeta_1$ are exactly canceled by those of $\widetilde \zeta_1$. Since $\lambda$ is graphical, we can take a linear homotopy $\lambda_t=(1-t)\lambda$ between $\lambda$ and the zero field, i.e $TL$. Let $R_t:T(T^*L) \to T(T^*L)$ be a family of symplectic bundle isomorphisms such that $R_t( \lambda) = \lambda_t$. Where $\widetilde \zeta = 0$ we have $R_1(\zeta_1)=TL$. Where $\widetilde \zeta \neq 0$ we have that $R_1(\zeta_1)$ is a rank 1 form 

%Observe then that $R_1(\zeta_1-\widetilde \zeta) = TL$, hence 

We also need the following elementary fact.

\begin{lemma}\label{lemma: symplectic covering} Let $f_t: \Lambda \to M$ be a Lagrangian isotopy, i.e. an exact regular homotopy of Lagrangian embeddings of a compact manifold $\Lambda$ into a symplectic manifold $M$ and for $i=0,1$, let $\gamma_i \subset TM$ be a Lagrangian plane field along $f_i$ which is transverse to $df_i(T \Lambda)$. Then there exists a compactly supported Hamiltonian isotopy $\varphi_t:M \to M$ such that $\varphi_t \circ f_0 = f_t$ and $d \varphi_1(\gamma_0)=\gamma_1$.
\end{lemma}

\begin{proof}
By taking a family of Weinstein neighborhoods for $f_t$ we reduce to the case $M=T^*\Lambda$,  $f_t=\text{id}_\Lambda$ and $\gamma_0=\nu$ (the vertical distribution). Since $\gamma_1$ is transverse to the zero section, we can think of $\gamma_1$ as family of quadratic forms on the fibres $\lambda_q : T^*_q\Lambda \to \bR$. Then the required Hamiltonian isotopy $\varphi_t$ is given by the quadratic Hamiltonian $H(q,p)=\lambda_q(p)$, cut off at infinity. 
\end{proof}

\begin{remark}
From the proof we also deduce the relative version: if $f_t=f_0$ and $\gamma_0=\gamma_1$ on $Op(A)$ for $A \subset \Lambda$ a closed subset, then we can demand that $\varphi_t=\text{id}_M$ on $Op(A)$.
\end{remark}

Proposition \ref{prop: near ridges} follows immediately from the following lemma.

\begin{lemma}\label{lemma: extension holonomic}
Let $\Lambda \subset T^*L$ be a ridgy Lagrangian, $R$ its ridge locus, and $\gamma$ a Lagrangian field. Suppose that $\gamma$ is homotopic to a Lagrangian field $\wh \gamma$ which is transverse to $\Lambda$. Then there exists a Hamiltonian isotopy $\Lambda_t$ of $\Lambda$  and a Lagrangian field $\widetilde{\gamma}$ which agrees with $\gamma$ on $Op(R)$ such that: \begin{itemize}
\item $\Lambda_1 \pitchfork \gamma$ on $Op(R)$.
\item $ \Lambda_1 \pitchfork \widetilde \gamma$ everywhere.
\item $\widetilde{\gamma}$ is homotopic to $\gamma$ by a homotopy fixed on $Op(R)$.
\end{itemize}
Moreover, if $A \subset \Lambda$ is a closed subset and the homotopy between $\gamma$ and $\wh \gamma$ is fixed on $Op(A)$, then the Hamiltonian isotopy and the homotopy between $\widetilde \gamma$ and $\gamma$ can both be chosen to be fixed on $Op(A)$.
\end{lemma}

\begin{proof}[Proof of Proposition \ref{prop: near ridges}]
Indeed, let $\Lambda \subset T^*L$ be the graphical ridgy Lagrangian corresponding to an integrable tectonic field $\widetilde \zeta$, as in Lemma \ref{lemma: model} and apply Lemma \ref{lemma: extension holonomic}.
% take a family of symplectic bundle isomorphisms $\Phi_t$ covering the homotopy of Lagrangian plane fields, which we may assume to be constant on the ridge locus of $\Lambda_1$ and set $G_t=\Phi_{t-1}(T \Lambda_1)$ to obtain the desired tangential rotation. 
\end{proof}

It therefore remains to prove Lemma \ref{lemma: extension holonomic}.

\begin{proof}[Proof of Lemma \ref{lemma: extension holonomic}] Denote by $R \subset \Lambda$ the ridge locus and denote $\Lambda \setminus R =Q_1 \cup Q_2 \cup \cdots \cup Q_m$, where each $\ol Q_j$ is a smooth Lagrangian submanifold with corners.

We will prove by induction that for each $k=0,1,\ldots, m$ there exists a  ridgy Lagrangian $\Lambda^k \subset T^*L$ satisfying the following properties.
\begin{itemize}

\item[(a)] $\Lambda^k$ is Hamiltonian isotopic to $\Lambda$
\item[(b)] $\Lambda^k \pitchfork \gamma$ in a neighborhood of $\partial Q^k_1 \cup \cdots \cup \partial Q^k_k$, where  $\ol Q^k_j \subset \Lambda^k$ corresponds to $\ol Q_j \subset \Lambda$ under the Hamiltonian isotopy.
\item[(c)] There exists a homotopy $\gamma^k_t$ of $\gamma^k_0=\gamma$, fixed in a neighborhood of $\partial Q^k_1 \cup \cdots \cup \partial Q^k_k$, such that $\Lambda^k \pitchfork \gamma^k_1$ everywhere.
\end{itemize}

\emph{Base case} $(k=0)$.  In this case we take $\Lambda^0=\Lambda$ so condition (a) is vacuous. Since $k=0$, so is (b). Finally, condition (c) is the hypothesis of Lemma \ref{lemma: extension holonomic}.

%hence we only need to verify (c). By assumption $\Lambda \pitchfork \widetilde \gamma$ for $\widetilde \gamma $ a Lagrangian distribution homotopic to $\gamma$. A homotopy $\gamma_t$ between $\gamma_0=\gamma$ and $\gamma_1=\widetilde \gamma$ can be covered by a family of symplectic bundle isomorphisms $S_t: T(T^*L) \to T(T^*L)$, so that $\gamma_t=S_t(\gamma)$. If we set $G_t=S_t^{-1}(T\Lambda)$, then $G_1 \pitchfork \gamma$ as required.

\emph{Inductive step} $(k \Rightarrow k+1)$. We apply the holonomic approximation theorem for 1-holonomic sections to the Lagrangian submanifold $\ol Q_{k+1}^k$ and the stratified subset $\partial Q^k_{k+1}$. Technically one should slightly enlarge $\ol Q_{k+1}$ so that $\partial Q_{k+1}$ sits in its interior but this will not affect the proof. The precise result we need is the h-principle for Lagrangian embeddings which are $D$-directed along a stratified subset, which is Theorem 1.20 in \cite{AG18a}. For our application we take $D$ to consist of all Lagrangian planes transverse to $\gamma$. Condition (c) of the inductive hypothesis implies that the Gauss map of $\ol Q^k_{k+1}$ is homotopic to a map with image in $D$, moreover this homotopy can be taken relative to a neighborhood of $\partial  Q^k_1 \cup \cdots \cup \partial Q^k_k$.

The output of the h-principle is a Hamiltonian isotopy $\varphi_t:T^*L \to T^*L$ such that $\varphi_1(\ol Q^k_{k+1}) \pitchfork \gamma$ in a neighborhood of $\varphi_1 (\partial Q^k_{k+1})$. Moreover, by the parametric version of the h-principle we may assume that $\varphi_t(\ol Q^k_{k+1})$ is transverse to $\gamma^k_{1-t}$ for all $t \in [0,1]$ in that same neighborhood. By Lemma \ref{lemma: symplectic covering} we may assume that $d\varphi_1(\gamma^k_1)=\gamma$ along $\partial Q^k_{k+1}$. Hence we have $\varphi_1(\Lambda) \pitchfork \gamma$ in a neighborhood of $\varphi_1(\partial Q^k_{k+1})$, i.e. also on the other side of the ridges outside of $\ol Q^k_{k+1}$. Moreover, by the relative version of the holonomic approximation theorem we can demand that $\varphi_t=\text{id}_{T^*L}$ in a neighborhood of $\partial Q^k_1 \cup \cdots \cup \partial Q^k_k $. It follows that if we set $\Lambda^{k+1}=\varphi_1(\Lambda^k)$, then conditions (a) and (b) are satisfied for $k+1$ instead of $k$.

It remains to verify condition (c). Consider the homotopy $\gamma^{k+1}_t$ of $\gamma$ which is given by the concatenation of first $\gamma^k_t$ and then $d\varphi_t(\gamma^k_1)$. The result is transverse to $\Lambda^{k+1}=\varphi_1(\Lambda^k)$ because $\gamma^k_1$ is transverse to $\Lambda^k$. Although this homotopy is constant in a neighborhood of $\partial Q^{k+1}_1 \cup \cdots \cup \partial Q^{k+1}_k$, it is not constant near $\partial Q^{k+1}_{k+1}$. To fix this we recall that $d \varphi_1(\gamma^k_1)=\gamma$ along $\partial Q^{k}_{k+1}$ and that $\varphi_t(\Lambda^k)$ is transverse to $\gamma^k_{1-t}$ along $\partial Q^{k+1}_k$. Since the space of linear Lagrangian planes transverse to a fixed linear Lagrangian plane is contractible, by parametrically interpolating between $d \varphi_t(\gamma^k_1)$ and $\gamma^k_{1-t}$ we can cancel out the concatenations using a cutoff function and thus deform the homotopy so that it is constant in a neighborhood of $\partial Q^{k+1}_{k+1}$. This completes the proof of the inductive step.

Finally, the relative form of the statement follows by applying the relative form of the holonomic approximation lemma at each stage of the induction.
%To finish the proof of Proposition \ref{prop: near ridges} we define the desired ridgy isotopy $L_t \subset T^*L$ of the zero section $L_0=L$ as the concatenation of two ridgy isotopies. The first one is a ridgy isotopy between $L$ and $\Lambda$, the graphical ridgy Lagrangian obtained by integrating $\widetilde \zeta$. The second one  is the Hamiltonian isotopy which one gets at the last stage $k=m$ of the above inductive process. For the resulting ridgy Lagrangian $L_1=\Lambda_m \subset T^*L$ we know that $\gamma$ is homotopic, relative to a neighborhood of the ridge locus of $L_1$, to a Lagrangian plane field transverse to $L_1$. Take a family of symplectic bundle isomorphisms $\Phi_t$ covering this homotopy of Lagrangian plane fields. We can take $\Phi_t$ to be constant in a neighborhood of the ridge locus of $L_1$. Then $G_t=\Phi_t^{-1}(TL_1)$ gives the required tangential rotation.
%Consider a first the homotopy of $\gamma$ which is given by $d\varphi_t^{-1}(\gamma)$ away from a neighborhood of $\varphi_1(\partial Q^k_{k+1})$, where we cut off the $t$ parameter so that the homotopy is constant in a smaller neighborhood of $\varphi_1(\partial Q^k_{k+1})$. The result is a Lagrangian field $\widetilde \gamma$ which is equal to $\gamma_1$ away from a neighborhood of $\varphi_1(\partial Q^k_{k+1})$ and interpolates 
\end{proof}

\begin{remark}
Since the holonomic approximation lemma holds in $C^0$-close form, at each step of the proof we can ensure $C^0$-closeness to the previous step and thus obtain that the resulting isotopy is $C^0$-small.
\end{remark}

\begin{remark}
Suppose that $K_1, K_2 \subset \Lambda$ are two disjoint compact subsets. Then given the hypotheses of Lemma \ref{lemma: extension holonomic} we may produce a Hamiltonian isotopy $\Lambda_t$ and a Lagrangian field $\widetilde \gamma$ such that the conclusion of the Lemma is satisfied on $Op(K_1)$ and both the isotopy $\Lambda_t$ and the homotopy between $\widetilde \gamma$ and $\gamma$ is constant on $Op(K_2)$. This follows simply by reparametrizing the time co-ordinate of the isotopy and homotopy using a function $\psi: \Lambda \to [0,1]$ which is 1 on $K_1$ and 0 on $K_2$.
\end{remark}

\subsection{Ridgification of wrinkles}

To conclude the proof of Theorem \ref{theorem: main theorem}, we need to take the ridgy Lagrangian produced by Proposition \ref{prop: near ridges} and further deform it in the complement of the ridge locus so that it becomes transverse to $\gamma$. To achieve this we use a wrinkling technique, namely the transversalization theorem for wrinkled Lagrangian embeddings.

\begin{theorem}[Theorem 5.1 of \cite{AG18b}]\label{thm: wrinkles} Let $\Lambda \subset M$ be a Lagrangian submanifold, $\gamma \subset TM$ a Lagrangian distribution, and $\gamma_t \subset TM$ a homotopy of Lagrangian distributions such that $\lambda_0=\lambda$ and $\Lambda \pitchfork \gamma_1$. Then there exists a $C^0$-small exact homotopy of wrinkled Lagrangian embeddings $\Lambda_t$ of $\Lambda$ such that $\Lambda_1 \pitchfork \gamma$. The result holds in relative form with respect to a closed subset $A \subset M$, i.e. if $\gamma_t$ is constant on $Op(A)$ then the homotopy of wrinkled Lagrangian embeddings can be taken to be constant on $Op(A)$. \end{theorem}

A wrinkled Lagrangian embedding is a smooth Lagrangian embedding outside of a disjoint union of codimension 1 contractible spheres $S \subset \Lambda$. Here and below, by a contractible sphere $S \subset \Lambda$ we mean a sphere that bounds an embedded disk.

Along each such sphere $S$ the embedding has cuspidal singularities of the form $\{p^2=q^3\} \times \bR^{n-1} \subset T^* \bR \times T^*\bR^{n-1}$, see Figure \ref{cuspidal}, with a codimension 1 equatorial sphere $\Sigma \subset S$ where the cuspidal singularities experience birth/death. The precise model near the equator is not important since, in the spirit of Entov \cite{En97}, we can surger away the birth/death singularities. More precisely, one can open up each sphere $S$ along its equator into two parallel spheres so that $\Lambda_1$ becomes a Lagrangian submanifold which is smooth away from the disjoint union of finitely many pairs of contractible parallel spheres $S_1 \cup S_2$ where the Lagrangian has cuspidal singularities. Moreover, this can be achieved while maintaining exactness and transversality to $\gamma$. The proof is simply to implant an explicit local model given by a generating function graphical over $\gamma$, just as in Section 6.1 of \cite{AG18b}, hence exactness and transversality are automatic. 

We note that this cuspidal Lagrangian can be smoothed so that the resulting smooth Lagrangian has fold tangencies on $S_1 \cup S_2$ with opposite Maslov co-orientations, known as \emph{double folds}. By performing this smoothing one deduces the h-principle for the simplification of caustics, which is Theorem 1.11 in \cite{AG18b}. However, it will be easier for us to work with cuspidal Lagrangians directly, implanting a local model for the ridgification of cusps.

\begin{remark}
Even though this will not be important, we note that the pairs of spheres $S_1 \cup S_2 \subset \Lambda$   could be nested, in the sense that we could have $A_1 \subset A_2$ for $A_i \simeq S^{n-1} \times [0,1]$, $i=1,2$, the codimension 0 annuli in $\Lambda$ corresponding to two pairs of parallel spheres.
\end{remark}

      \begin{figure}[h]
\includegraphics[scale=0.4]{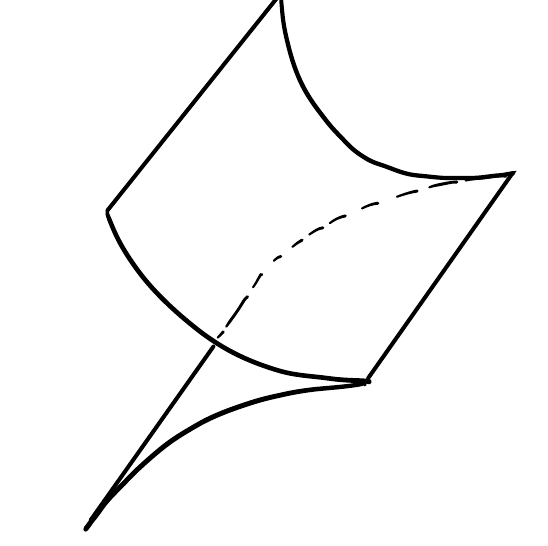}
\caption{A cuspidal Lagrangian singularity is a stabilization of the standard semi-cubical cusp $\{p^2=q^3\} \subset T^*\bR$.}
\label{cuspidal}
\end{figure}

We deduce the following consequence of Theorem \ref{thm: wrinkles}.

\begin{corollary}\label{thm: h-p} Let $\Lambda \subset M$ be a Lagrangian submanifold, $\gamma \subset TM$ a Lagrangian distribution and $\gamma_t$ a homotopy of Lagrangian distributions such that $\gamma_0=\gamma$ and $\Lambda \pitchfork \gamma_1$. Then there exists a $C^0$-small ridgy isotopy $\Lambda_t$ of $\Lambda$ such that $\Lambda_1 \pitchfork \gamma$. The result holds in relative form with respect to a closed subset $A \subset M$, i.e. if $\gamma_t$ is constant on $Op(A)$ then the ridgy isotopy can be taken to be constant on $Op(A)$. \end{corollary}

\begin{proof}
We need to modify the wrinkled Lagrangian embedding $\Lambda_1$ produced by Theorem \ref{thm: wrinkles} to make it ridgy. We first resolve the equators of the wrinkles into cuspidal singularities as above while maintaining  exactness and transversality to $\gamma$. Next, by a local interpolation at the level of generating functions we can replace these cuspidal singularities with stabilizations of order 1 ridges $\{q=|p|\}$ while still maintaining exactness and transversality to $\gamma$.  

Explicitly, take a cut-off function $\sigma:[0,\infty)\to[0,1]$ which is equal to $1$ on $[0,\frac12]$, equal to $0$ outside $[0,1]$. Consider the generating function with variable $q \geq 0$ given by
\[ z = \frac{\varepsilon}{2} \sigma \left( \frac{q}{\varepsilon} \right) q^2 + \frac{2}{5} \big( 1 - \sigma \left( \frac{q}{\varepsilon} \right) \big) q^{5/2} . \]
The function $\pm z$ generates a ridgy Lagrangian $\Lambda_\varepsilon \subset T^*\bR$ for $\varepsilon>0$. Note that
\[ \frac{\partial z}{\partial q} =  \frac{1}{2} \sigma'\left(\frac{q}{\varepsilon}\right)q^2 +  \varepsilon \sigma\left(\frac{q}{\varepsilon} \right) q - \frac{2}{5\varepsilon}\sigma'\left(\frac{q}{\varepsilon}\right)q^{5/2} +  \big( 1- \sigma\left(\frac{q}{\varepsilon} \right) \big) q^{3/2} \]
\[ = q^{3/2} +  \Big( \frac{1}{2} \sigma'\left(\frac{q}{\varepsilon}\right)q^2 + \varepsilon \sigma\left(\frac{q}{\varepsilon} \right) q - \frac{2}{5\varepsilon}\sigma'\left(\frac{q}{\varepsilon}\right)q^{5/2}  - \sigma\left(\frac{q}{\varepsilon} \right)  q^{3/2}  \Big) \quad = \quad  q^{3/2} + O(\varepsilon) \]
since we have the obvious bound $$ \sigma^{(k)}\left(\frac{q}{\varepsilon}\right)q^r \leq \| \sigma \|_{C^k} \, \,  \varepsilon^r.$$ 
Hence as $\varepsilon \to 0$, the ridgy Lagrangian $\Lambda_\varepsilon$ gets $C^0$-close to the cuspidal Lagrangian generated by $z=\pm \frac{2}{5}q^{5/2}$. Furthermore, we have \[ \frac{\partial^2 z}{\partial q^2} = \frac{3}{2}q^{1/2} + O(\sqrt{\varepsilon})\]
since $\partial^2 z / \partial q^2 - \frac{3}{2}q^{1/2}$ is equal to 
\[ \frac{1}{2 \varepsilon}\sigma''\left(\frac{q}{\varepsilon} \right) q^2 + \sigma'\left(\frac{q}{\varepsilon}\right)q + \sigma'\left(\frac{q}{\varepsilon} \right) q + \varepsilon \sigma \left( \frac{q}{\varepsilon} \right) - \frac{2}{5 \varepsilon^2} \sigma''\left( \frac{q}{\varepsilon} \right) q^{5/2} - \frac{1}{\varepsilon} \sigma'\left(\frac{q}{\varepsilon} \right) q^{3/2} - \frac{1}{\varepsilon} \sigma'\left( \frac{q}{\varepsilon} \right) q^{3/2} - \frac{3}{2} \sigma\left(\frac{q}{\varepsilon}\right) q^{1/2} \]
and we can bound each term as before. Hence in fact as $\varepsilon \to 0$ the ridgy Lagrangian $\Lambda_\varepsilon$ gets $C^1$ close to the cuspidal Lagrangian generated by $z = \pm \frac{2}{5}z^{5/2}$.

%Explicitly, take a cut-off function $\sigma:[0,\infty)\to[0,1]$ which is equal to $1$ on $[0,\frac12]$, equal to $0$ outside $[0,1]$ and has non-positive derivative. Define  the function $$\phi_\varepsilon(p)= \frac32\left(1-\sigma\left(\frac{|p|}{\varepsilon}\right)\right)p^{2/3}+\frac{1}{2\ve}\sigma\left(\frac{|p|}{\varepsilon}\right)\sign(p)p^2$$  and set
%$$R(\varepsilon):=\left\{q =\frac{\p \phi_\varepsilon(p)}{\p p}\right\}.$$
%For $\ve=0$ we have $R(0)=\{p^2=q^3\}$ and for any $\ve>0$ small $R(\ve)$ is a ridgy Lagrangian with a ridge at the origin, which gets infinitely sharp as $\ve \to 0$. If $C \subset \Lambda_1$ is a component of the cuspidal locus, we can apply the above interpolation parametrically over $x \in C$ to obtain a deformation $\Lambda_1(\ve)$ of $\Lambda_1(0)=\Lambda_1$ such that for any $\ve>0$ small $\Lambda_1(\ve)$ is a ridgy Lagrangian. Moreover, by taking $\ve>0$ arbitrarily small we can ensure that the Gauss map of $\Lambda_1(\ve)$ is $C^0$-close to that of $\Lambda_1$. Hence for $\ve>0$ small enough $\Lambda_1(\ve)$ is still transverse to $\gamma$. 

Therefore by parametrically implanting a 1-dimensional model we can replace the exact homotopy of wrinkled Lagrangian embeddings $\Lambda_t$ with a ridgy isotopy which at time 1 is transverse to $\gamma$. This ridgy isotopy consists of an earthquake isotopy on the preimage of the cuspidal locus of $\Lambda_1$ in $\Lambda$ followed by an ambient Hamiltonian isotopy, the existence of which is guaranteed because we ensured exactness at every stage by working at the level of generating functions.  \end{proof}

\begin{remark}
The ridge locus of the ridgy Lagrangian $\Lambda_1$ produced by Corollary \ref{thm: h-p} therefore consists of a disjoint union of parallel contractible spheres, which may be nested. 
\end{remark}

\begin{proof}[Proof of the main theorem \ref{theorem: main theorem}]
We apply Proposition \ref{prop: near ridges} to $L$ and $\gamma|_L$ inside a Weinstein neighborhood $U$ of $L$ in $M$. Indeed, $U$ is symplectomorphic to $T^*_{\leq \delta}L = \{ || p || < \delta \}$, where $\delta>0$ and we use an auxiliary Riemannian metric on $L$. If the resulting ridgy Lagrangian is sufficiently $C^0$-close to the zero section then it will remain in this Weinstein neighborhood, hence can be viewed as a ridgy Lagrangian in $M$. We then apply the relative (with respect to the ridge locus) version of Corollary  \ref{thm: h-p} to the output of Proposition \ref{prop: near ridges}.  \end{proof}

      \begin{figure}[h]
\includegraphics[scale=0.6]{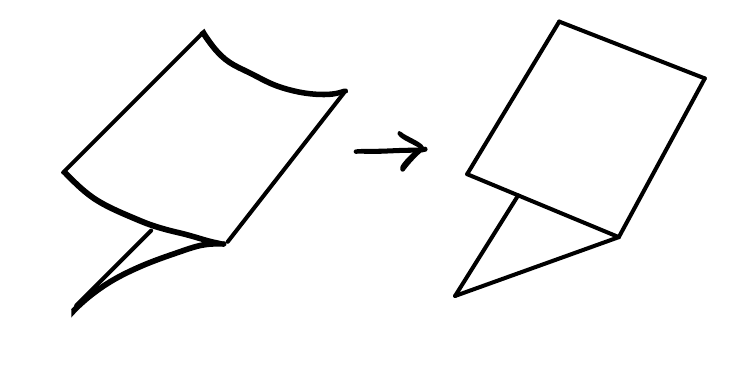}
\caption{Replacing a cuspidal singularity with an order 1 ridge.}
\label{swap}
\end{figure}

\begin{remark}
The ridge locus of the ridgy Lagrangian solving the transversalization problem consists of the fault locus of the formal solution together with a union of parallel contractible spheres.
\end{remark}

\section{Adapted version}\label{section: adapted version}

\subsection{Adapted transversalization}

Let $L$ be an $n$-dimensional compact manifold with boundary and corners. Recall that $L$ has a corner of order $k \leq n$ at $x \in L$ if there is a neighborhood of $x$ in $L$ diffeomorphic to a neighborhood of the origin in $[0,1)^k \times \R^{n-k}$. We denote the locus of order $k$ corners by $\partial_k L_k$. 
The closure $P$ of a connected component of $\p_kL$ is called a  boundary $k$-face. For $k=1$ we will more simply call $P$ a boundary face.

If $Q$ is a $k$-face, then there is an embedded collar neighborhood $Q \times [0,1)^k \subset L$. We consider the germs of these collars as part of the structure. In particular, near each point $x \in \partial_k L$ we have canonical collar coordinates $x=(y,t)$, where $y \in \partial_k L$ and $t=(t_1, \ldots , t_k) \in [0,1)^k$. Note that in a neighborhood of $x$ we have $\partial_k L$ cut out by $t_1 = \cdots = t_k=0$. More generally, for $j\leq k$ the components of $\partial_j L$ whose closure contains $x$ are given by setting exactly $j$ of the coordinates $t_i$ equal to zero.

\begin{definition}
A Lagrangian field $\lambda$ on $L$ (possibly tectonic) is said to be:
\begin{itemize}
\item \emph{horizontally adapted} if $\lambda = \lambda_k \times \tau^k \subset T^*(\partial_k L) \times (T^*\cI)^k$ near each $x \in \partial_k L$, $k=1,2,\ldots, n$.
\item \emph{vertically adapted} if $\lambda = \lambda_k \times \nu^k \subset T^*(\partial_k L) \times (T^* \cI)^k$ near each $x \in \partial_k L$, $k=1,2,\ldots, n$.
\end{itemize}
\end{definition}
\begin{definition} A Lagrangian submanifold (possibly ridgy) $\Lambda \subset T^*L$ is said to be \emph{adapted} if $\Lambda=\Lambda_k \times \cI^k \subset T^*(\partial_k L) \times (T^*\cI)^k$ near each $x \in \partial_kL$, $k=1,2,\ldots ,n$. A (possibly ridgy) isotopy of Lagrangian submanifolds $\Lambda_t \subset T^*L$ is said to be \emph{adapted} if each $\Lambda_t$ is adapted.
\end{definition}

      \begin{figure}[h]
\includegraphics[scale=0.5]{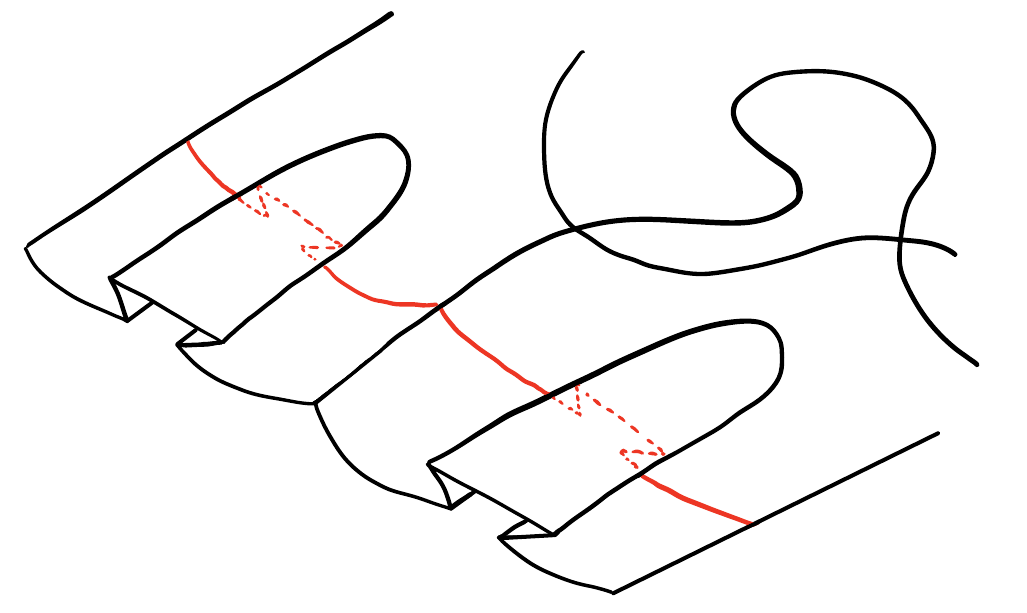}
\caption{An adapted ridgy Lagrangian near the collar.}
\label{collaredridgy}
\end{figure}

\begin{remark}
If $\Lambda \subset T^*L$ is adapted, then $T\Lambda$ is horizontally adapted.
\end{remark}

We can now state the adapted version of our main theorem \ref{theorem: main theorem}.

\begin{theorem}\label{theorem: main theorem adapted}
For any vertically adapted Lagrangian field $\gamma \subset T^*L$ there exists an adapted ridgy isotopy $L_t$ of the zero section $L_0=L$ such that $L_1 \pitchfork \gamma$. 
\end{theorem}

Theorem \ref{theorem: main theorem adapted} also holds in $C^0$-close and relative forms. 
The proof of Theorem \ref{theorem: main theorem adapted} proceeds just like in the unadapted case: first we construct a formal solution, then we align it and finally we integrate it. We must argue that the same proof works while ensuring that all the objects are adapted to the collar structure at each step.

\subsection{Adapted formal transversalization}

The adapted version of the formal transversalization theorem \ref{theorem: formal transversalization} reads as follows.

\begin{theorem} For any vertically adapted Lagrangian field $\gamma$ there exists a horizontally adapted tectonic field $\lambda$ such that $\lambda \pitchfork \gamma$.
\end{theorem}

The extension form \ref{theorem: formal transversalization extension} of the result also has its adapted version.

\begin{theorem}\label{theorem: formal transversalization extension adapted} Let $\gamma$ be a vertically adapted Lagrangian field and $\zeta$ a horizontally adapted tectonic field. For any two disjoint compact subsets $K_1,K_2 \subset L$ there exists a horizontally adapted tectonic field $\wh \zeta$ such that the following properties hold.
\begin{itemize}
\item $\wh \zeta$ is $C^0$-close to $\zeta$,
\item $\wh \zeta \pitchfork \gamma$ on $Op(K_1)$.
\item $\wh \zeta = \zeta$ on $Op(K_2)$.
\end{itemize}
\end{theorem}

To prove Theorem \ref{theorem: formal transversalization extension adapted} one inductively applies Lemma \ref{lemma: local extension}, just as in the proof of Theorem \ref{theorem: formal transversalization extension}. The only difference is that before constructing $\wh \zeta$ in the interior of $L$ one constructs $\wh \zeta$ in a neighborhood of $\partial L$,  inductively over the strata $\partial_k L$. Start with the deepest stratum $\partial_n L$ where there is nothing to prove. At each step of the induction one has a horizontally adapted tectonic field $\wh \zeta$ defined over a neighborhood of $\bigcup_{j\geq k } \partial_j L$ which satisfies the required properties. To continue with the induction one chooses a cover of $\partial_{k-1}L$ by balls and applies Lemma \ref{lemma: local extension} in the manifold $\partial_{k-1}L$, one ball at a time. Multiplying the resulting tectonic field by the horizontal distribution in the collar direction provides the extension and so the induction can continue. Once the horizontally adapted tectonic field $\wh \zeta$ has been built in a neighborhood of $\partial L$ it can be extended to the rest of $L$ as in the unadapted case. 

\subsection{Adapted aligned formal transversalization}

The adapted version of the aligned analogue Theorem \ref{theorem: aligned formal transversalization} of Theorem \ref{theorem: formal transversalization} reads as follows.

\begin{theorem}\label{aligned formal transversalization adapted} For any vertically adapted Lagrangian field $\gamma$ there exists a horizontally adapted aligned tectonic field $\lambda$ such that $\lambda \pitchfork \gamma$.
\end{theorem}

More generally, we have the adapted version of the aligned extension result Theorem \ref{theorem: aligned formal transversalization extension}.

\begin{theorem}\label{theorem: aligned formal transversalization extension adapted} Let $\gamma$ be a vertically adapted Lagrangian field and $\zeta$ a horizontally adapted aligned tectonic field. For any two disjoint compact subsets $K_1,K_2 \subset L$ there exists a horizontally adapted aligned tectonic field $\wh \zeta$ and a vertically adapted Lagrangian field $\wh \gamma$ homotopic to $\gamma$ such that the following properties hold.
\begin{itemize}
\item $\wh \zeta$ is $C^0$-close to $\zeta$,
\item $\wh \zeta \pitchfork \wh \gamma$ on $Op(K_1)$.
\item $\wh \zeta = \zeta$ on $Op(K_2)$.
\end{itemize}
Moreover, we can assume that the homotopy between $\gamma$ and $\wh \gamma$ is through vertically adapted fields and is constant on $Op(K_2)$.
\end{theorem}

Theorem \ref{theorem: aligned formal transversalization extension adapted} follows from the same local model for changing the homotopy class of the ridge directions which we used to align the ridge directions in the unadapted case. Indeed, as in Section \ref{section: changing} we first reduce to the homotopically aligned version of Theorem \ref{theorem: aligned formal transversalization extension adapted}. To prove the homotopically aligned version we can start making the necessary local modifications to the $\wh \zeta$ produced by Theorem \ref{theorem: formal transversalization extension adapted} along the boundary $\partial L$ first, then once we have a horizontally adapted solution near $\partial L$ we can extend to the interior as in the unadapted case. To construct the horizontally adapted solution $\wh \zeta$ near $\partial L$ we work inductively over the strata $\partial_k L$, starting with the deepest one $\partial_nL$ in which there is nothing to prove. Whenever we need to adjust the homotopy class of the ridge directions for the tectonic field $\lambda_k$ in $\partial_k L$, we choose a domain $\Omega \subset L$ and a form $\nu$ as in the model \ref{section: the model} which are adapted to the collar structure, i.e. given as a product in the collar co-ordinates. Then not only is the modified tectonic field still adapted, but the homotopy of $\gamma$ is by construction through vertically adapted fields.

\subsection{Adapted integration}

Finally we show how to integrate the $\lambda$ produced by Theorem \ref{aligned formal transversalization adapted} so that the resulting ridgy Lagrangian remains adapted. In fact this follows easily from the parametric versions of the holonomic approximation and wrinkling results which are used in the unadapted case. First observe that since the aligned tectonic field $\lambda$ is horizontally adapted, the introduction of integrable ridges in Lemma \ref{lemma: model} can be achieved with respect to coordinates that are compatible with the collar structure. Hence the resulting integrable tectonic field is horizontally adapted. 

Next we turn to the adapted analogue of Proposition \ref{prop: near ridges}. We recall that the holonomic approximation lemma for 1-holonomic sections \cite{AG18a} holds in parametric form, and moreover holds relative to a closed subset of the parameter space. Hence we can apply this result inductively over the strata $\partial_kL$ so that at each stage of the induction the conclusion of the proposition holds in a neighborhood of $\bigcup_{j \geq k }  \partial_j L$ and moreover such that the resulting ridgy Lagrangian and homotopy of $\gamma$ are adapted in this neighborhood. At the last stage of this inductive process we obtain the desired adapted ridgy Lagrangian in a neighborhood of $\partial L$, which can then be extended to the interior of $L$ as in the unadapted case. 

To conclude we turn to the application of Theorem \ref{thm: wrinkles} in the adapted setting and finish the proof of Theorem \ref{theorem: main theorem adapted}. The $C^0$-approximation result for wrinkled Lagrangian embeddings also holds in parametric form, but only relative to a subset where the embedding is smooth. Therefore, when applying the result in a component of the stratum $\partial_k L$ one will need to let the cuspidal singularities die out as you move away from this component, but there is a homotopically canonical way of doing so since by construction the cuspidal singularities always come in parallel spheres which can be cancelled against each other, see Figure \ref{birthdeath} 

      \begin{figure}[h]
\includegraphics[scale=0.5]{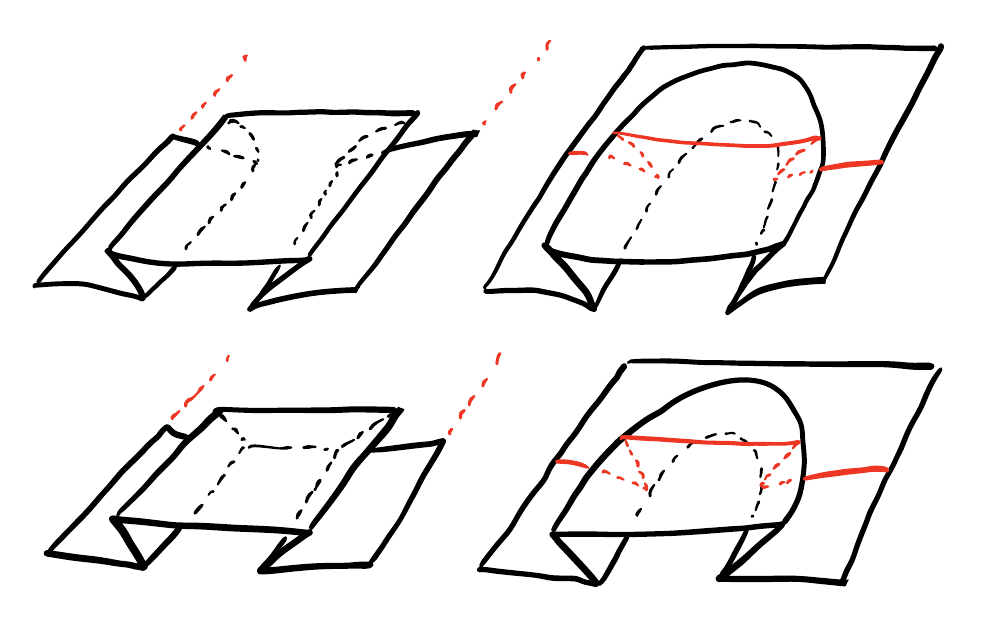}
\caption{Starting with product cuspidal singularities on parallel spheres, we let them die out and then replace the cusps with ridges. }
\label{birthdeath}
\end{figure}

Note that in the insertion of this birth/death local model for the cuspidal singularities
we may lose transversality. Therefore, after replacing these cuspidal singularities with ridges, but before we proceed to the next stage of the induction, we must achieve transversality in a neighborhood of these new ridges. This can be achieved using holonomic approximation just as in Section \ref{section: holonomic}. We can then use wrinkling as before in the complement and replace the cuspidal singularities with ridges to complete the proof.

\end{document}